\documentclass[12pt,a4paper]{article}
\usepackage[margin=1.1cm]{geometry}
\usepackage{amsmath,amsthm,amsfonts,amssymb,amscd,cite,graphicx}
\usepackage{latexsym}
\usepackage{enumitem}

\usepackage[usenames, dvipsnames]{color}
\definecolor{mygray}{gray}{0.6}

\usepackage{titlesec}
\titleformat{\section}
{\normalfont\fontsize{12}{15}\bfseries}{\thesection}{1em.}{}

\newtheorem{corollary}{Corollary}[section]
\newtheorem{lemma}{Lemma}[section]

\newtheorem{theorem}{Theorem}[section]

\allowdisplaybreaks[4]

\let\oldbibliography\thebibliography
\renewcommand{\thebibliography}[1]{%
  \oldbibliography{#1}%
  \setlength{\itemsep}{-2pt}%
}

\begin{document}

\baselineskip=0.20in

\makebox[\textwidth]{%
\hglue-15pt
\begin{minipage}{0.6cm}	
\vskip9pt
\end{minipage} \vspace{-\parskip}
\begin{minipage}[t]{6cm}

\end{minipage}
\hfill
\begin{minipage}[t]{6.5cm}
\end{minipage}}
\vskip36pt

\noindent
{\large \bf Enumeration of consecutive patterns in flattened Catalan words}\\

\noindent
 Mark Shattuck\\

\noindent
\footnotesize {\it Department of Mathematics, University of Tennessee,
37996 Knoxville, TN, USA\\
Email: mshattuc@utk.edu}\\

\noindent
(\footnotesize Received: Day Month 201X. Received in revised form: Day Month 201X. Accepted: Day Month 201X. Published online: Day Month 201X.)\\

\setcounter{page}{1} \thispagestyle{empty}

\baselineskip=0.20in

\normalsize

\begin{abstract}
A Catalan word $w$ is said to be \emph{flattened} if the subsequence of $w$ obtained by taking the first letter of each weakly increasing run is nondecreasing.  Let $\mathcal{F}_n$ denote the set of flattened Catalan words of length $n$, which has cardinality $\frac{3^{n-1}+1}{2}$ for all $n \geq 1$.  In this paper, we consider the distribution of several consecutive patterns on $\mathcal{F}_n$.  Indeed, we find explicit formulas for the generating functions of the joint distribution on $\mathcal{F}_n$ of several trios of patterns, along with an auxiliary parameter.  As special cases of these formulas, we obtain the generating function  for the distribution of all consecutive patterns of length two or three. The following equivalences with regard to being identically distributed on $\mathcal{F}_n$ arise when comparing the various generating functions and may be explained bijectively: $112\approx122$ and $211\approx221\approx231$.  In addition, explicit expressions are found for the total number of occurrences on $\mathcal{F}_n$ of each pattern of length two or three as well as for the number of avoiders of each pattern. These results can be obtained as special cases of our more general formulas for the generating functions, but may be explained combinatorially as well, the arguments of which are featured herein.
\end{abstract}

\section{Introduction}

A \emph{Catalan} word is an integral sequence $\pi=\pi_1\cdots \pi_n$ such that $1 \leq \pi_{i+1} \leq \pi_i+1$ for $1 \leq i \leq n-1$ with $w_1=1$.  Let $\mathcal{C}_n$ denote the set of Catalan words of length $n$.  It is well-known that $|\mathcal{C}_n|=C_n$ for all $n$, where $C_n=\frac{1}{n+1}\binom{2n}{n}$ denotes the $n$-th Catalan number; see, e.g., \cite[Exercise~80]{Stan2}. Catalan words arise in the context of various combinatorial structures such as rooted binary trees and Dyck paths (see, e.g., \cite{Mak} and \cite{Shat}, respectively) and have also been studied in connection with the exhaustive generation of Gray codes \cite{MVaj} for growth-restricted words. The enumeration of Catalan words with regard to certain restrictions was initiated by Baril et al.~\cite{BKV2}, who studied the distribution of descents on members of $\mathcal{C}_n$ avoiding a single classical pattern of length two or three. These avoidance results on $\mathcal{C}_n$ were extended to pairs of such patterns in \cite{BKV1}, and later work in this direction included consecutive patterns \cite{RamR}, pairs of partial order relations \cite{BRam} and certain statistics \cite{CMR,MRT} defined in terms of the bargraph representation of a Catalan word.

We now recall the concept of a \emph{flattened} Catalan word introduced in \cite{BHR}.  To do so, first recall that by a (weakly) \emph{increasing run} within a word $w=w_1\cdots w_n$, it is meant a  maximal nondecreasing sequence of consecutive letters of $w$.  That is, an increasing run corresponds to a pair of indices $1 \leq i \leq j \leq n$ such that $w_i\leq w_{i+1}\leq\cdots \leq w_j$ wherein the following two conditions hold: (a) $i=1$ or $i>1$ with $w_{i-1}>w_i$, and (b) $j=n$ or $j<n$ with $w_{j+1}<w_j$.  For example, if $w$ is a word of length twelve given by $w=347655895244$,  then $w$ has five increasing runs given by $347, 6, 5589, 5, 244$.

A \emph{flattened} Catalan word $w$ is then defined as a member of $\mathcal{C}_n$  in which the subsequence of $w$ obtained by taking the first letter in each of its increasing blocks is nondecreasing.  For example, the Catalan word $\pi=123323234334 \in \mathcal{C}_{12}$ is flattened since the letters starting its increasing runs from left to right are $1,2,2,3$, whereas $\pi=121232341223$ is not since its increasing runs start $1,1,2,1$.  Let $\mathcal{F}_n$ for $n\geq 1$ denote the subset of $\mathcal{C}_n$ consisting of its flattened members.  It was shown in \cite{BHR} that $|\mathcal{F}_n|=\frac{3^{n-1}+1}{2}$ for $n \geq 1$, with this sequence corresponding to \cite[A007051]{Sloane}. For example, we have $|\mathcal{F}_6|=122$, the ten excluded members of $\mathcal{C}_6$ comprising
$$\mathcal{C}_6-\mathcal{F}_6=\{112321,122321,123211,123212,123221,123231,123321,123421,123431,123432\},$$
with $|\mathcal{C}_6|=C_6=132$. Note that the increasing run lengths of any member of $\mathcal{F}_n$ are all at least two, except for possibly the last.  In particular, no member of $\mathcal{F}_n$ can have two consecutive descents.  Let $\mathcal{F}=\cup_{n\geq1}\mathcal{F}_n$ denote the set of all Catalan words of positive length, on which we will consider several joint distributions of statistics below.

Let $\tau=\tau_1\cdots\tau_m$ be a positive integral sequence whose distinct entries comprise the set $[\ell]=\{1,\ldots,\ell\}$ for some $\ell \geq 1$.  Then the sequence $\pi=\pi_1\cdots\pi_n$ is said to \emph{contain} $\tau$ as a \emph{consecutive pattern} if there exists a string of consecutive letters of $\pi$ that is order-isomorphic to $\tau$.  Consecutive patterns are also referred to as \emph{subword patterns}, or simply as \emph{subwords}.  Thus, $\pi$ contains $\tau$ as a subword when there exists an index $i \in [n-m+1]$ such that $\pi_i\pi_{i+1}\cdots\pi_{i+m-1}$ is isomorphic to $\tau$ and \emph{avoids} $\tau$ as a subword if no such $i$ exists.  For example, the sequence $\pi=4135753346$ contains three occurrences of 123 as a subword,  as witnessed by the strings 135, 357 and 346, and a single occurrence of the 121 subword corresponding to the string 575. Note that the occurrences of a subword pattern need not be mutually disjoint.  The sequence $\pi$ avoids 132 as a subword, though it contains subsequences that are isomorphic to the pattern.  The problem of counting occurrences of subwords is one that has been undertaken on a variety of finite discrete structures, represented sequentially, such as $k$-ary words \cite{BM}, compositions \cite{MSir}, set partitions \cite{MSY}, non-crossing partitions \cite{MSh} and Catalan words \cite{Shat}.

Here, we will address the problem of counting the members of $\mathcal{F}_n$ according to the number of occurrences of subword patterns of length two or three.  This extends recent work of Baril et al.~\cite{BHR} enumerating members of $\mathcal{F}_n$ according to certain statistics involving the relative sizes of adjacent entries, such as peaks, valleys and runs of ascents/descents, where the corresponding generating function formulas are found as are the respective totals on $\mathcal{F}_n$.  Note that there are thirteen subwords of length two or three in all and we achieve our count of each such subword pattern by finding five joint distributions on $\mathcal{F}$ involving three patterns each.

The problem of enumerating flattened discrete structures was first considered by Callan \cite{Callan}, who counted flattened partitions of $[n]$ avoiding a single classical pattern of length three.  A flattened partition is obtained by removing each pair of parentheses enclosing its blocks written in standard form and considering the resulting permutation expressed in one-line notation.  This idea had actually been applied much earlier by Carlitz \cite{Car} to the cycles of a permutation expressed in standard cycle form in defining a certain kind of inversion statistic on $\mathcal{S}_n$.   Other flattened structures that have subsequently been enumerated with respect to various restrictions include permutations \cite{MSW}, involutions \cite{MSh0}, Stirling permutations \cite{BEF} and parking functions \cite{EHM}. In \cite{NRB}, the set of distinct permutations that arise as flattened partitions of $[n]$ is considered, and recurrences along with the exponential generating function formula are deduced for the distribution of the statistic tracking the number of increasing runs.  Finally, statistics comparable to the one defined by Carlitz on permutations have been introduced on derangements \cite{GR} and set partitions \cite{FH} using different flattened versions of these structures wherein orderings other than the standard ones are applied to the cycles or blocks, respectively.

The organization of this paper is as follows. In the second section, we derive a formula for the joint distribution of the ascent, descent and level statistics on $\mathcal{F}_n$, together with the length and an auxiliary parameter to be described below that was useful in the derivation.  We then consider the comparable distribution involving the three subwords of length three that end in a level, namely, 122, 211 and 111.  For both distributions, we determine an explicit formula for the total number of occurrences of each of the three patterns on $\mathcal{F}_n$ for $n \geq 1$ as well as a formula  enumerating the number of avoiders in $\mathcal{F}_n$.  We provide a direct counting argument for each of these expressions, which may also be obtained as special cases of the general distributions on $\mathcal{F}$ involving all of the parameters. In the third section, a similar treatment is given for the analogous joint distributions on $\mathcal{F}$ involving the following three trios of subword patterns: $112/121/221$, $123/231/221$ and $112/212/312$.  The following equivalences in distribution on $\mathcal{F}_n$ for all $n \geq 1$ arose by comparing special cases of the general formulas: $112\approx122$ and $211\approx221\approx231$.  These equivalences may subsequently be explained bijectively.

To find the generating function $F(x,y)$ for the joint distribution of the three subword patterns $\alpha$, $\beta$ and $\gamma$ on $\mathcal{F}_n$ for all $n \geq 1$, we make use of various arrays involving the length and a statistic tracking the number of distinct letters within the terminal increasing run of a Catalan word (here, $x$ and $y$ mark the length of a member of $\mathcal{F}$ and this auxiliary statistic value, respectively).  This will lead to functional equations relating $F(x,y)$ to the $y$-partial derivative of $F(x,y)$ evaluated at $y=1$ and involving three extra parameters (one for each of the subwords $\alpha$, $\beta$ and $\gamma$).  We introduce a technique for solving the functional equations that arise in this manner, which may be applicable to a broader class of equations.  Given a subword $\tau$, let $F_\tau(x;q)$ denote the generating function for the distribution of $\tau$ on $\mathcal{F}_n$ for $n \geq 1$; that is,
$$F_\tau(x;q)=\sum_{n\geq1}\left(\sum_{\pi\in\mathcal{F}_n}q^{\#\tau(\pi)}\right)x^n,$$
where $\#\tau(\pi)$ denotes the number of occurrences of $\tau$ in $\pi$. Let $\text{tot}_n(\tau)$ for $n\geq 1$ denote the total number of occurrences of $\tau$ in all the members of $\mathcal{F}_{n}$. As special cases of our main results (see Theorems \ref{Axyth}, \ref{Bxyth}, \ref{Cxyth}, \ref{Dxyth} and \ref{Exyth} below) involving joint distributions of three patterns, we obtain explicit formulas for $F_\tau(x;q)$ and $\text{tot}_n(\tau)$ where $n\geq 2$ for all $\tau$ of length two or three, which are listed in Table \ref{tab1}.  The final column in Table \ref{tab1} gives the OEIS reference when it occurs for the distinct sequences occurring amongst $\text{tot}_n(\tau)$ for the various patterns $\tau$.

\begin{table}[htp]
{\small\begin{center}
        \begin{tabular}{|l|l|l|l|l|}\hline
            Class & $\tau$ & $F_{\tau}(x;q)$ & $\text{tot}_n(\tau)$ & OEIS  \\\hline\hline\hline
            &&&&\\[-8pt]
            1 & 11 & $\frac{x(1-(1+q)x)}{1-2(1+q)x+q(2+q)x^2}$ & $\frac{(n-1)(3^{n-2}+1)}{2}$  & A082133 \\\hline
            &&&&\\[-8pt]
            2 & 12  & $\frac{x(1-(1+q)x)}{1-2(1+q)x+(1+q+q^2)x^2}$ & $\frac{(n-1)(3^{n-1}+1)}{4}$ &\\\hline
            3 & 21  & $\frac{x(1-2x)}{1-4x+(4-q)x^2}$ & $\frac{(n-1)(3^{n-2}-1)}{4}$ & A261064\\\hline
            &&&&\\[-8pt]
             4& 111  & $\frac{x(1+(1-q)x)(1-(1+q)x-(1-q)x^2)}{1-2(1+q)x-(2-4q-q^2)x^2+2q(1-q)x^3}$ & $\text{tot}_{n-1}(11)$ & \\\hline
            5 & 112&&&\\
              &122 & $\frac{x(1-2x+(1-q)x^2)}{1-4x+(5-2q)x^2-3(1-q)x^3+(1-q)^2x^4}$ & $\text{tot}_{n-1}(12)$&\\\hline
              6 & 121  & $\frac{x(1-2x)}{1-4x+(4-q)x^2-2(1-q)x^3}$ & $\frac{(n+1)3^{n-3}+n-3}{4}$ & \\\hline
            7 & 123  & $\frac{x(1+(1-q)x)(1-(1+q)x-(1-q)x^2)}{1-2(1+q)x-(2-4q-q^2)x^2+2(1-q^2)x^3+(1-q)^2x^4}$ &$(n-2)3^{n-3}$ & A027471\\\hline
             8 & 211&&&\\
               &221 & && \\
             &231 & $\frac{x(1-2x)}{1-4x+3x^2+(1-q)x^3}$ &$\text{tot}_{n-1}(21)$ & \\\hline
             &&&&\\[-8pt]
             9 & 212  & $\frac{x(1-3x+(3-q)x^2-2(1-q)x^3)}{(1-x)(1-4x+(4-q)x^2-2(1-q)x^3)}$ &$\text{tot}_{n-1}(21)$ & \\\hline
             &&&&\\[-8pt]
             10 & 312 & $\frac{x(1-3x+2x^2+(1-q)x^3)}{(1-x)(1-4x+3x^2+(1-q)x^3)}$ &$\frac{(n-5)3^{n-3}+n-1}{4}$ & A212337\\\hline
        \end{tabular}
    \end{center}}
\caption{$F_\tau(x;q)$ and $\text{tot}_n(\tau)$ for all subwords patterns $\tau$ of length two or three.}\label{tab1}
\end{table}

\section{Two joint distributions on flattened Catalan words}

In this section, we establish formulas for the joint distributions on $\mathcal{F}$ involving the length and two triples of subword patterns by solving a pair of multi-parameter functional equations with derivative terms.  The technique illustrated is potentially applicable to other classes of functional equations.

In order to determine functional equations satisfied by the aforementioned joint distributions, we consider a further parameter as follows.   Let us refer to the rightmost increasing run within a word $w$ as its \emph{terminal run}.  Define $\text{trun}(w)$ to be the number of distinct letters appearing in the terminal run of $w$.  For example, we have $\text{trun}(w)=2$ for the word given by $w=347655895244$ as its terminal run is 244.  To aid in finding functional equations and recursions for arrays enumerating members of $\mathcal{F}$ according to the number of occurrences of certain subword patterns, we will consider the $\text{trun}-1$ parameter on members of $\mathcal{F}$.

In particular, we will be interested in determining explicitly multivariate distributions $F(x,y)=F(x,y;p,q,r)$ of the form
$$F(x,y)=\sum_{\pi \in \mathcal{F}}x^{|\pi|}y^{\text{trun}(\pi)-1}p^{\#\alpha(\pi)}q^{\#\beta(\pi)}r^{\#\gamma(\pi)},$$
where $x,y,p,q,r$ are indeterminates, $\alpha,\beta,\gamma$ are subword patterns and $|\pi|$ denotes the length of $\pi$. Then $[x^n]F(x,y)$ for each $n \geq 1$ gives the joint distribution of the $\alpha$, $\beta$ and $\gamma$ subwords on $\mathcal{F}_n$, together with $\text{trun}-1$ (marked by $p$, $q$, $r$ and $y$  respectively). Note that $F(x,1;1,1,1)=\frac{x(1-2x)}{(1-x)(1-3x)}$, which is the generating function for $|\mathcal{F}_n|=\frac{3^{n-1}+1}{2}$ for $n \geq 1$. As will be seen, it is often more convenient algebraically to deal with the parameter $\text{trun}-1$ rather than $\text{trun}$, especially when working with terms involving the $y$-partial derivative of $F(x,y)$.  Further, as seen in Subsections \ref{ssec3.2} and \ref{ssec3.3}, the $y=0$ (and also $y=1$) case of $F(x,y)$ will play an instrumental role in determining $F(x,y)$ for $y$ in general. Note that taking $y$ and two of $\{p,q,r\}$ to be unity in $F(x,y;p,q,r)$ gives the generating function $F_\tau(x;z)$ for $\tau \in \{\alpha,\beta,\gamma\}$, by the definitions.  For example, taking $y=q=r=1$, we have $F(x,1;p,1,1)=F_\alpha(x;p)$.

We consider first in this section the joint distribution on $\mathcal{F}$ for ascents/descents/levels and then the comparable distribution involving the patterns $122/211/111$.

\subsection{Distribution of ascents, descents and levels}

Recall that an \emph{ascent}, \emph{descent} or \emph{level} within a word $w=w_1\cdots w_n$ refers to an index $i \in [n-1]$ such that $w_i<w_{i+1}$, $w_i>w_{i+1}$ or $w_i=w_{i+1}$, respectively.  Let $\text{asc}(w)$, $\text{des}(w)$ and $\text{lev}(w)$ denote respectively the number of ascents, descents and levels in the word $w$, equivalently, the number of $12$, $21$ and $11$ subwords in $w$.   Define $A(x,y)=A(x,y;p,q,r)$ by
$$A(x,y)=\sum_{\pi \in \mathcal{F}}x^{|\pi|}y^{\text{trun}(\pi)-1}p^{\text{asc}(\pi)}q^{\text{des}(\pi)}r^{\text{lev}(\pi)}.$$
Then $A(x,y)$ satisfies the following functional equation.

\begin{lemma}\label{lemA(x,y)}
We have
\begin{equation}\label{lemA(x,y)e1}
(1-x(py+r))A(x,y)=x+qx\frac{\partial}{\partial y}A(x,y)\mid_{y=1}.
\end{equation}
\end{lemma}
\begin{proof}
We apply the \emph{symbolic} (see, e.g., \cite{FS}) method and enumerate the members of $\mathcal{F}$ according to the relative sizes of the final two letters.  Let $\pi \in \mathcal{F}-\mathcal{F}_1$ be given by $\pi=\pi'ab$ for some $a,b \geq 1$.  If $a<b$, then $b=a+1$ and the weight of such members of $\mathcal{F}$ is given by $pxyA(x,y)$, as both the $\text{asc}$ and $\text{trun}$ values are increased by one whenever $b$ is appended to $\pi'a$ in this case.  If $a=b$, then we get a contribution of $rxA(x,y)$.

So suppose $a>b$ within $\pi \in \mathcal{F}$, in which case we consider the value $m=\text{trun}(\pi'a)$.  Then $m \geq 2$ and $b \in [a-m+1,a-1]$, since $b \leq a-m$ would violate membership of $\pi$ in $\mathcal{F}$ (as $a-m+1$ is seen to be the largest descent bottom in $\pi'a$, assuming it contains at least one descent, with the descent bottoms required to be nondecreasing).  Thus, to obtain a member of $\mathcal{F}$ ending in a descent, first consider marking (the leftmost appearance of) some letter $b \in [a-m+1,a-1]$ occurring within the terminal run of $\pi'a$.  The generating function for the \emph{weight} of such marked members of $\mathcal{F}$ is given by
$$\sum_{n\geq2}x^n\sum_{m=2}^n(m-1)a_{n,m}=\frac{\partial}{\partial y}A(x,y)\mid_{y=1},$$
upon considering all possible values of $m$, where $a_{n,m}=[x^ny^{m-1}]A(x,y)$.  We then erase the marking and append a copy of the marked letter $b$ to $\pi'a$ so as to obtain $\pi=\pi'ab$ with $a>b$.  This has the effect of multiplying the generating function by $qx$ since both the length and the number of descents are increased by one.  Further, the resulting words $\pi$ each have trun value one (as they end in a descent), and hence their contribution towards $A(x,y)$ has $y$-coefficient zero. The preceding observations taken together thus imply that the weight of all members of $\mathcal{F}$ of the form $\pi'ab$ for some $a>b \geq 1$ is given by
$qx\frac{\partial}{\partial y}A(x,y)\mid_{y=1}$. Combining the cases above, and adding the contribution from the single member of $\mathcal{F}_1$, we then get
$$A(x,y)=x+pxyA(x,y)+rxA(x,y)+qx\frac{\partial}{\partial y}A(x,y)\mid_{y=1},$$
which may be rewritten as \eqref{lemA(x,y)e1}.
\end{proof}

We now solve the functional equation \eqref{lemA(x,y)e1} with derivative term.  In so doing, we describe a technique which could be applied to other such functional equations.

\begin{theorem}\label{Axyth}
The generating function that enumerates the members of $\mathcal{F}$ jointly according to length, $\text{trun}-1$, asc, des and lev (marked by $x$, $y$, $p$, $q$ and $r$, respectively) is given by
\begin{equation}\label{Axythe1}
A(x,y)=\frac{x(1-(p+r)x)^2}{(1-(r+py)x)(1-2(p+r)x+((p+r)^2-pq)x^2)}.
\end{equation}
\end{theorem}
\begin{proof}
First let $A(x,y)=\sum_{n\geq1}\sum_{m=1}^na_{n,m}x^ny^{m-1}$, and we use \eqref{lemA(x,y)e1} to obtain recurrences for the coefficients $a_{n,m}$.  Comparing the coefficient of $x^ny^{m-1}$ on both sides of \eqref{lemA(x,y)e1}, and considering whether $m>1$ or $m=1$, we get
\begin{align}
a_{n,m}=pa_{n-1,m-1}+ra_{n-1,m}, \qquad 2 \leq m \leq n,\label{Axythe2}\\
a_{n,1}=ra_{n-1,1}+q\sum_{j=2}^{n-1}(j-1)a_{n-1,j}, \qquad n \geq 2, \label{Axythe3}
\end{align}
with initial value $a_{1,1}=1$. Note that the recurrences \eqref{Axythe2} and \eqref{Axythe3} for the coefficients of the unknown function $A(x,y)$ may also be explained combinatorially in this case.

We now define two auxiliary sequences $u_n$ and $v_n$ which will enable us to determine a formula for $A(x,y)$.  Let $u_n=\sum_{j=2}^n(j-1)a_{n,j}$ for $n\geq 2$, with $u_1=0$, and let $v_n=\sum_{j=1}^na_{n,j}$ for $n\geq 1$. Multiplying both sides of \eqref{Axythe2} by $y^{m-1}$, summing over $2 \leq m \leq n$, differentiating both sides of the resulting equation with respect to $y$ and setting $y=1$ gives
\begin{align}
u_n&=\sum_{m=2}^n(m-1)a_{n,m}=p\sum_{m=2}^n(m-1)a_{n-1,m-1}+r\sum_{m=2}^{n-1}(m-1)a_{n-1,m}\notag\\
&=p\sum_{m=2}^n(m-2)a_{n-1,m-1}+p\sum_{m=2}^{n}a_{n-1,m-1}+ru_{n-1}=p\sum_{m=2}^{n-1}(m-1)a_{n-1,m}+pv_{n-1}+ru_{n-1}\notag\\
&=(p+r)u_{n-1}+pv_{n-1}, \qquad n \geq 2. \label{Auneq1}
\end{align}
Also, adding equations \eqref{Axythe2} and \eqref{Axythe3}, we have
\begin{align}
v_n&=\sum_{m=1}^na_{n,m}=r\sum_{m=1}^{n-1}a_{n-1,m}+p\sum_{m=2}^na_{n-1,m-1}+q\sum_{j=2}^{n-1}(j-1)a_{n-1,j}\notag\\
&=(p+r)v_{n-1}+qu_{n-1}, \qquad n \geq 2, \label{Auneq2}
\end{align}
with $u_1=0$ and $v_1=1$

Let $U=U(x)=\sum_{n\geq2}u_nx^n$ and $V=V(x)=\sum_{n\geq1}v_nx^n$.  Then we seek a formula for $U$ since $U(x)=\frac{\partial}{\partial y}A(x,y)\mid_{y=1}$, by the definitions, and hence
\begin{equation}\label{Axythe4}
A(x,y)=\frac{x+qxU(x)}{1-(py+r)x},
\end{equation}
by \eqref{lemA(x,y)e1}.  Multiplying both sides of \eqref{Auneq1} and \eqref{Auneq2} by $x^n$, and summing over all $n \geq 2$, yields
\begin{align*}
U&=(p+r)xU+pxV,\qquad \qquad V=x+(p+r)xV+qxU.
\end{align*}
Solving the preceding system of equations in $U$ and $V$ for $U$ gives
$$U=\frac{px^2}{1-2(p+r)x+((p+r)^2-pq)x^2}.$$
Substituting this expression for $U$ into \eqref{Axythe4}, and simplifying, implies that the solution to the functional equation \eqref{lemA(x,y)e1} is given by the formula in \eqref{Axythe1}, as desired.
\end{proof}

Taking $y=1$ in \eqref{Axythe1} gives
\begin{equation}\label{A(x,1)}
A(x,1)=\frac{x(1-(p+r)x)}{1-2(p+r)x+((p+r)^2-pq)x^2},
\end{equation}
from which one can obtain the generating functions of the individual distributions for ascents, descents or levels on $\mathcal{F}_n$ by setting two of $\{p,q,r\}$ equal to unity.  Thus, we have
\begin{align}
F_{12}(x;q)&=A(x,1;q,1,1)=\frac{x(1-(q+1)x)}{1-2(q+1)x+(q^2+q+1)x^2},\label{F12x}\\
F_{21}(x;q)&=A(x,1;1,q,1)=\frac{x(1-2x)}{1-4x-(q-4)x^2},\label{F21x}\\
F_{11}(x;q)&=A(x,1;1,1,q)=\frac{x(1-(q+1)x)}{1-2(q+1)x+q(q+2)x^2}.\label{F11x}
\end{align}
Further, setting $p=q=r=1$ in \eqref{A(x,1)}, or equivalently $q=1$ in each of \eqref{F12x}--\eqref{F11x}, yields $\frac{x(1-2x)}{(1-x)(1-3x)}$, as required.  Moreover, the generating function for the distribution of $\text{trun}-1$ on $\mathcal{F}_n$ for $n \geq 1$ is obtained by setting $p=q=r=1$ in \eqref{Axythe1}:
\begin{equation}\label{trungf}
A(x,y;1,1,1)=\sum_{n\geq1}\left(\sum_{\pi \in \mathcal{F}_n}y^{\text{trun}(\pi)-1}\right)x^n=\frac{x(1-2x)^2}{(1-x-xy)(1-4x+3x^2)}.\
\end{equation}

Let $\text{tot}_n(\tau)$ denote the total number of occurrences of a subword pattern $\tau$ for $n \geq 1$.  Differentiation of \eqref{F12x}--\eqref{F11x} with respect to $q$, setting $q=1$ and extracting the coefficient of $x^n$ yields the following result.

\begin{corollary}\label{Axyc1}
If $n \geq 1$, then
$$\text{tot}_n(12)=\frac{(n-1)(3^{n-1}+1)}{4},\qquad\text{tot}_n(21)=\frac{(n-1)(3^{n-2}-1)}{4},\qquad\text{tot}_n(11)=\frac{(n-1)(3^{n-2}+1)}{2}.$$
Moreover, the sum of the $\text{trun}-1$ values of all the members of $\mathcal{F}_n$ is given by $\frac{3^{n-1}-1}{2}$ for $n \geq 1$.
\end{corollary}

In Subsection \ref{combproofs}, we provide combinatorial proofs of the formulas in Corollary \ref{Axyc1}. \medskip

\noindent{\bf Remarks:} We have $\text{tot}_n(21)=A261064[n-2]$ for $n \geq 3$ and $\text{tot}_n(11)=A082133[n-1]$ for $n\geq 1$, where $A\#\#\#\#\#\#[m]$ denotes the OEIS sequence parameterized as in the indicated entry.  Indeed, we have that each adjacency $ab$ where $a>b$ within a member of $\mathcal{F}_n$ is part of a larger string of letters having the form $(a-1)a^\ell b$ for some $\ell \geq 1$.  Such strings are referred to as \emph{peaks} in \cite{BHR}, where a generating function formula was found for their distribution of $\mathcal{F}_n$.  Hence, the number of descents coincides with this peak statistic on $\mathcal{F}_n$, which accounts for the total number of occurrences of each being the same on $\mathcal{F}_n$. \medskip

\subsection{Distribution of $\#122$, $\#211$ and $\#111$}

In this subsection, we consider the comparable joint distribution on $\mathcal{F}$ involving the 122, 211 and 111 subwords.  Define $B(x,y)=B(x,y;p,q,r)$ by
$$B(x,y)=\sum_{\pi \in \mathcal{F}}x^{|\pi|}y^{\text{trun}(\pi)-1}p^{\#122(\pi)}q^{\#211(\pi)}r^{\#111(\pi)}.$$
Then $B(x,y)$ satisfies the following functional equation.

\begin{lemma}\label{122l1}
We have
\begin{equation}\label{122l1e1}
(1-(r+y)x-(p-r)x^2y)B(x,y)=x+x^2(1-r)+x(1+(q-r)x)\frac{\partial}{\partial y}B(x,y)\mid_{y=1}.
\end{equation}
\end{lemma}
\begin{proof}
Let $\pi=\pi'ab \in \mathcal{F}-\mathcal{F}_1$.  First note that the weight of $\pi$ where $a<b$, and hence $b=a+1$, is given by $xyB(x,y)$.  If $a>b$, then the weight of such $\pi$ is $xB'$, where $B'=\frac{\partial}{\partial y}B(x,y)\mid_{y=1}$, upon reasoning as in the second paragraph in the proof of \eqref{lemA(x,y)e1} above.  So assume $a=b$, in which case, we consider the size of the antepenultimate letter $c$ of $\pi$.  If $c<a$, and hence $a=c+1$, then we get a contribution towards the weight of $px^2yB(x,y)$, whereas, if $c>a$, then the contribution is given by $qx^2B'$.  Note that both trun and $\#122$ increase by one in the former case, whereas only $\#211$ increases in the latter, with trun equal to one (due to the descent $c,a$).

So assume $c=a$, i.e., $\pi$ ends in $a,a,a$ for some $a\geq 1$.  Then the contribution towards the generating function $B(x,y)$ of such $\pi$ is given by $rx\omega$, where $\omega$ denotes the weight of all members of $\mathcal{F}$ ending in a level.  By subtraction, we have
$$\omega=B(x,y)-x-xyB(x,y)-xB',$$
upon excluding the single member of $\mathcal{F}_1$ as well as all the members of $\mathcal{F}$ ending in either an ascent or a descent (which accounts for the second and third subtracted terms, respectively).  Combining the various cases above, and including the contributions from the Catalan words $1$ and $1^2$ (which are seen not to be part of any of the previous cases), we get
$$B(x,y)=x+x^2+xyB(x,y)+xB'+px^2yB(x,y)+qx^2B'+rx\omega,$$
where $\omega$ is as stated, which leads to \eqref{122l1e1}.
\end{proof}

We now solve the functional equation \eqref{122l1e1} by applying the technique developed for \eqref{lemA(x,y)e1}.

\begin{theorem}\label{Bxyth}
The generating function that enumerates the members of $\mathcal{F}$ jointly according to length, $\text{trun}-1$, $\#122$, $\#211$ and $\#111$  (marked by $x$, $y$, $p$, $q$ and $r$, respectively) is given by
\begin{equation}\label{Bxythe1}
B(x,y)=\frac{x(1+(1-r)x)(1-(1+r)x-(p-r)x^2)^2}{(1-(r+y)x-(p-r)x^2y)\left((1-(1+r)x-(p-r)x^2)^2-x^2(1+(p-r)x)(1+(q-r)x)\right)}.
\end{equation}
\end{theorem}
\begin{proof}
Define the array $b_{n,m}$ by $B(x,y)=\sum_{n\geq1}\sum_{m=1}^nb_{n,m}x^ny^{m-1}$.  Let $u_n$, $v_n$, $U=U(x)$ and $V=V(x)$  be as in the proof of \eqref{Axythe1} above except now they are defined in conjunction with $b_{n,m}$. Comparing coefficients of $x^ny^{m-1}$ on both sides of \eqref{122l1e1} implies for $n \geq 3$:
\begin{align}
b_{n,m}&=rb_{n-1,m}+b_{n-1,m-1}+(p-r)b_{n-2,m-1}, \qquad 2 \leq m \leq n,\label{Bxythe2}\\
b_{n,1}&=rb_{n-1,1}+u_{n-1}+(q-r)u_{n-2}, \label{Bxythe3}
\end{align}
with the initial conditions $b_{1,1}=b_{2,1}=b_{2,2}=1$.  Proceeding as before, we then obtain from \eqref{Bxythe2} and \eqref{Bxythe3} the following system of recurrences for $n \geq 3$:
\begin{align}
u_n&=(1+r)u_{n-1}+v_{n-1}+(p-r)(u_{n-2}+v_{n-2}),\label{Bxythe4}\\
v_n&=u_{n-1}+(1+r)v_{n-1}+(q-r)u_{n-2}+(p-r)v_{n-2}, \label{Bxythe5}
\end{align}
with $u_1=0,v_1=u_2=1,v_2=2$. Multiplying both sides of \eqref{Bxythe4} and \eqref{Bxythe5} by $x^n$, and summing over $n \geq 3$, then gives
\begin{align*}
U&=x^2+x(1+r+(p-r)x)U+x(V-x)+(p-r)x^2V,\\
V&=x+2x^2+x(1+r+(p-r)x)V-(1+r)x^2+x(1+(q-r)x)U,
\end{align*}
which implies
\begin{align*}
U&=\frac{x(1+(p-r)x)}{1-x(1+r+(p-r)x)}V, \qquad \qquad V=\frac{x+(1-r)x^2+x(1+(q-r)x)U}{1-x(1+r+(p-r)x)}.
\end{align*}
Solving for $U$ in the preceding system gives
$$U=\frac{x^2(1+(1-r)x)(1+(p-r)x)}{\left(1-(1+r)x-(p-r)x^2\right)^2-x^2(1+(p-r)x)(1+(q-r)x)}.$$
Finally, substituting this expression for $U$ into \eqref{122l1e1}, and simplifying, leads to the desired formula for $B(x,y)$.
\end{proof}

Differentiation of the generating functions for the individual subword distributions on $\mathcal{F}_n$, and comparing with Corollary \ref{Axyc1}, yields the following result.

\begin{corollary}\label{Bxyc1}
If $n \geq 2$, then $\text{tot}_n(122)=\text{tot}_{n-1}(12)$, $\text{tot}_n(211)=\text{tot}_{n-1}(21)$ and $\text{tot}_n(111)=\text{tot}_{n-1}(11)$,
where $\text{tot}_m(12)$, $\text{tot}_m(21)$ and $\text{tot}_m(11)$ are as given in Corollary \ref{Axyc1}.
\end{corollary}
\noindent These results are apparent combinatorially; for example, inserting an extra copy of the ascent top in a marked ascent within a member of $\mathcal{F}_{n-1}$ is seen to yield an arbitrary member of $\mathcal{F}_n$ containing at least one $122$ such that one of the occurrences of $122$ is marked.

Let $\mathcal{F}_n(\tau)$ denote the subset of $\mathcal{F}_n$ whose members avoid the subword $\tau$ and let $f_n(\tau)=|\mathcal{F}_n(\tau)|$. To obtain the generating function of the sequence $f_n(\tau)$ for $n \geq 1$, one can set $q=0$ in $F_\tau(x;q)$.  Doing this for the generating functions $F_\tau(x;q)$ where $\tau \in \{122,211,111\}$, whose formulas are special cases of \eqref{Bxythe1}, and extracting the coefficient of $x^n$ in each case yields the following result.

\begin{corollary}\label{Bxyc2}
We have
\begin{align}
f_n(122)&=\sum_{r=0}^{\lfloor\frac{n-1}{2}\rfloor}\binom{n+r}{3r+1}, \qquad n \geq 1,\label{Bxyc2e1}\\
f_n(211)&=\sum_{j=1}^n\sum_{r=0}^{\lfloor\frac{j-1}{2}\rfloor}\binom{j-1}{2r}\binom{n-r-1}{j-r-1}, \qquad n \geq 1,\label{Bxyc2e2}\\
f_n(111)&=\sum_{k=0}^{\lfloor\frac{n}{2}\rfloor}\binom{n-k}{k}2^{n-k-2}, \qquad n \geq 3.  \label{Bxyc2e3}
\end{align}
\end{corollary}

\subsection{Combinatorial proofs}\label{combproofs}

In this subsection, we provide counting arguments for Corollaries \ref{Axyc1} and \ref{Bxyc2}.  \medskip

\noindent\emph{Proof of Corollary \ref{Axyc1}:}\medskip

Before proceeding, we will need a bijective proof of the following result, which may also be obtained by setting $q=0$ in \eqref{F11x}.

\begin{lemma}\label{11avoidlem}
If $n\geq2$, then the number of members of $\mathcal{F}_{n}$ without levels is given by $2^{n-2}$.
\end{lemma}
\begin{proof}
Let $A=A_n$ denote the set of $\{0,1\}$-words of length $n-1$ starting with 0 and $B=B_n$ the subset of $\mathcal{F}_n$ whose members contain no levels.  If $\pi=0^{n-1}$, then let $\pi'=12\cdots n$.  Otherwise, we represent $\pi \in A$ by
\begin{equation}\label{11avoidleme1}
\pi=0^{k_1-\ell_1}1^{\ell_1}\cdots0^{k_r-\ell_r}1^{\ell_r}0^t,
\end{equation}
for some $r \geq 1$, where $1 \leq \ell_i<k_i$ for each $i \in [r]$ and $t \geq0$.  In this case, we let $\pi'$ be a word on the alphabet of positive integers containing exactly $r+1$ increasing runs $I_0,I_1,\ldots,I_r$ defined as follows:
\begin{align*}
I_0&=12\cdots k_1,\\
I_a&=(s_a-a+1)(s_a-a+2)\cdots(s_a-a+k_{a+1}), \quad 1 \leq a \leq r-1,\\
I_r&=(s_r-r+1)(s_r-r+2)\cdots(s_r-r+t+1),
\end{align*}
where $s_a=\sum_{j=1}^a(k_j-\ell_j)$ for $1 \leq a \leq r$, with $s_0=0$.  Note that the first letter of $I_{a+1}$, namely, $s_{a+1}-a$, is strictly less than the last letter of $I_a$, which is $s_a-a+k_{a+1}$, for each $0 \leq a \leq r-1$.  Further, we have that the subsequence of $\pi'$ consisting of the initial 1 and its descent bottoms is nondecreasing, as $\ell_i<k_i$ for all $i$ implies $1 \leq s_1 \leq s_2-1\leq \cdots\leq s_r-r+1$.

Since $\sum_{a=0}^r|I_a|=(k_1+\cdots+k_r)+t+1=n$, we have $\pi' \in \mathcal{F}_n$ for all $\pi \in A$, by the preceding observations.  Indeed, since $\pi'$ is seen to contain no levels, we have $\pi' \in B$ for all $\pi$.  Further, note that the number of runs of 1 in $\pi$ equals one less than the number of increasing runs in $\pi'$.  Thus, to reverse the mapping $\pi \mapsto \pi'$, we first decompose $\rho \in B$ into its increasing runs.  This allows for one to determine the number of runs of 1 in its inverse image $\pi \in A$, expressed as in \eqref{11avoidleme1}, and subsequently the parameters $k_i$ and $\ell_i$ for each $i \in [r]$, along with $t$.  Thus, the mapping $\pi \mapsto \pi'$ is reversible and hence provides a bijection between $A$ and $B$, which implies $|B|=2^{n-2}$, as desired.
\end{proof}

Using the prior result, one can explain the formula $|\mathcal{F}_n|=\frac{3^{n-1}+1}{2}$ for $n \geq 1$ combinatorially as follows.  By a \emph{run} of a letter (as distinguished from an increasing run), we mean a maximal string of the form $b^\ell$ for some letter $b$ and $\ell \geq 1$.   By a \emph{skeleton} of a word $w=w_1w_2\cdots$, we mean the subsequence of $w$ obtained by deleting all letters of $w$ except for those starting a run of a letter.  Consider the length $j$ of the skeleton of a member of $\mathcal{F}_n$, where $j \in [n]$.  We augment the skeleton by increasing the various run lengths so as to obtain a member of $\mathcal{F}_n$, thereby adding a total of $n-j$ letters distributed over $j$ possible runs.  Note that there are $2^{j-2}$ possibilities for the skeleton if $j \geq 2$, by Lemma \ref{11avoidlem}.  Considering all possible $j$ then gives
\begin{align*}
|\mathcal{F}_n|&=1+\sum_{j=2}^{n}2^{j-2}\binom{(n-j)+j-1}{j-1}=1+(1/2)\sum_{j=1}^{n-1}2^j\binom{n-1}{j}=1+(1/2)(3^{n-1}-1)\\
&=(1/2)(3^{n-1}+1),
\end{align*}
where the ``+1'' appearing before the summation accounts for the member of $\mathcal{F}_n$ consisting of all $1$'s.

We are now in position to provide combinatorial explanations of the formulas from Corollary \ref{Axyc1}, where we may assume $n \geq 3$.\medskip

\noindent {\bf $\text{tot}_n(11)$:} A level within a member of $\mathcal{F}_n$ may be obtained by considering the letter $x$ in any position within any member of $\mathcal{F}_{n-1}$, inserting a copy of $x$ to directly follow it and then marking the resulting level.  Since the position may be chosen independently of the member of $\mathcal{F}_{n-1}$, we get $(n-1)|\mathcal{F}_{n-1}|=\frac{(n-1)(3^{n-2}+1)}{2}$ total levels in the members of $\mathcal{F}_n$. \medskip

\noindent {\bf $\text{tot}_n(21)$:} From the proof of Lemma \ref{11avoidlem}, we have that the number of descents in $\rho \in B$ is equal to the number of runs of $1$ in $\rho^*\in A$ for all $\rho$, where $\rho^*$ denotes the inverse image of $\rho$ with respect to the prime mapping on $B$. In particular, the total number of descents in $B$ equals the number of runs of 1 in $A$, and we seek to enumerate the latter.

Now observe that there are $(n-3)2^{n-4}$ runs of 1 in $A$ in which the final 1 in the run corresponds to position $i$ for some $i \in [2,n-2]$.  To see this, put $0,1,0$ for $\pi_1,\pi_i,\pi_{i+1}$ within $\pi=\pi_1\cdots \pi_{n-1} \in A$, fill in the other entries of $\pi$ arbitrarily with $0$'s and $1$'s ($2^{n-4}$ choices if $n\geq 4$) and then mark the run of 1 terminating at position $i$ in the resulting word.  This generates $(n-3)2^{n-4}$ marked members of $A$ wherein an internal run of 1 is marked.  To this, we add the number of terminal runs of 1 in $A$, which is seen to be $2^{n-3}$, upon putting $0$ and $1$ for $\pi_1$ and $\pi_{n-1}$, respectively, and filling in the other positions of $\pi$ arbitrarily.  This implies that $A$ has a total of $(n-3)2^{n-4}+2^{n-3}=(n-1)2^{n-4}$ runs of 1, and hence the same number of descents in $B$.

Since augmenting a skeleton of length $j\geq 3$ with $n-j$ letters, added to the runs as described above, does not change the number of descents, we have
$$\text{tot}_n(21)=\sum_{j=3}^n(j-1)2^{j-4}\binom{n-1}{j-1}=\frac{n-1}{4}\sum_{j=3}^n2^{j-2}\binom{n-2}{j-2}=\frac{(n-1)(3^{n-2}-1)}{4},$$
as desired. \medskip

\noindent {\bf $\text{tot}_n(12)$:} From the proof of Lemma \ref{11avoidlem}, we see that an ascent in $\rho \in B$ corresponds to a letter $y$ in $\rho^* \in A$ such that $y$ is not the final 1 within a run of 1 in $\rho^*$.  By subtraction and the proof of the formula for $\text{tot}_n(21)$, there are $(n-1)2^{n-2}-(n-1)2^{n-4}=3(n-1)2^{n-4}$ such letters $y$ in $A$.  Thus, there are $3(j-1)2^{j-4}\binom{n-1}{j-1}$ ascents altogether in the members of $\mathcal{F}_n$ whose skeletons have length $j$ for each $j \in[3,n]$, which is three times the corresponding number of descents.  Considering all possible $j$ and adding to this, $n-1$, to account for the ascents coming from the members of $\mathcal{F}_n$ whose skeleton has length two, we get
$$\text{tot}_n(12)=n-1+3\text{tot}_n(21)=n-1+\frac{3(n-1)(3^{n-2}-1)}{4}=\frac{(n-1)(3^{n-1}+1)}{4}.$$

\noindent {\bf $\text{trun}-1$:} Let $\mathcal{F}_n^*$ denote the set of marked members $\pi \in \mathcal{F}_n$ wherein one of the distinct letters, not the smallest, within the terminal increasing run of $\pi$ is designated.  Note that $|\mathcal{F}_n^*|$ is seen to equal the sum of the $\text{trun}-1$ parameter values taken over all the members of $\mathcal{F}_n$. To show $|\mathcal{F}_n^*|=\frac{3^{n-1}-1}{2}$, it suffices to define a bijection between $\mathcal{F}_n^*$ and $\mathcal{F}_n-\{1^n\}$, where $n\geq2$.  To do so, let $\alpha\in \mathcal{F}_n^*$ be decomposed as $\alpha=J_1J_2\cdots J_s$ for some $s \geq 1$, where the $J_i$ denote the increasing runs of $\alpha$.  We write out $J_s$ explicitly as $J_s=a_1^{k_1}\cdots a_t^{k_t}$ for some $t \geq 2$, where $a_{i+1}=a_i+1$ for $1 \leq i \leq t-1$ and $k_i \geq 1$ for each $i \in [t]$.  Suppose that the designated letter of $\alpha$ is $a_p$ for some $p \in [2,t]$.

If $p=t$, then let $f(\alpha)$ be given by
\begin{equation}\label{combtrune1}
f(\alpha)=1^{k_t}J_1'\cdots J_{s-1}'(J_s-a_t^{k_t})',
\end{equation}
where $w'=w+1$ for a sequence $w$ and $J_s-a_t^{k_t}$ is obtained by deleting the terminal string $a_t^{k_t}$ from $J_s$.  If $2 \leq p \leq t-1$, then let
\begin{equation}\label{combtrune2}
f(\alpha)=1^{k_p}2^{k_{p+1}}\cdots(t-p+1)^{k_t}J_1\cdots J_{s-1}(J_s-a_p^{k_p}a_{p+1}^{k_{p+1}}\cdots a_t^{k_t}).
\end{equation}
One may verify that $f(\alpha) \in \mathcal{F}_n$ for all $\alpha \in \mathcal{F}_n^*$.  Further, the all $1$'s member of $\mathcal{F}_n$ is not in the range of $f$ since $t\geq 2$ with $p \in [2,t]$ implies that the section $J_s-a_t^{k_t}$ in \eqref{combtrune1} and also $2^{k_{p+1}}\cdots(t-p+1)^{k_t}$ in \eqref{combtrune2} are nonempty.  Note that the range of $f$ in \eqref{combtrune1} is the subset of $\mathcal{F}_n-\{1^n\}$ whose members contain only one run of 1, whereas in \eqref{combtrune2}, it consists of those members containing two or more runs of 1.   Therefore, the mapping $f$ may be reversed by considering whether a member of $\mathcal{F}_n-\{1^n\}$ contains one or more runs of 1.  Thus, $f$ provides the desired bijection and completes the proof of Corollary \ref{Axyc1}. \hfill \qed \medskip

\noindent\emph{Proof of Corollary \ref{Bxyc2}:}\medskip

To show \eqref{Bxyc2e1}, consider forming members of $\mathcal{F}_n(122)$ by adding $n-j$ letters to members of $\mathcal{F}_j(11)$ for some $j \in [n]$, as described above. In order to avoid 122, we must be careful not to increase the run lengths of any ascent tops within a skeleton.  Suppose $\rho \in \mathcal{F}_j(11)$, with $\text{des}(\rho)=r$.  Then there are $r+1$ places (namely, directly following one of the $r$ descent bottoms or after the initial $1$) within $\rho$ where extra letters may be added to obtain members $\pi \in \mathcal{F}_n(122)$ from $\rho$.  Thus, there are $\binom{n-j+r}{r}$ possible $\pi$ that arise from each such $\rho$.

To enumerate $\rho$ where $\text{des}(\rho)=r$, note that such $\rho$ are in one-to-one correspondence with members $\alpha \in A_{j}$ containing exactly $r$ runs of $1$, where the set $A_{j}$ is as in the proof of Lemma \ref{11avoidlem}.  Such $\alpha$ are synonymous with integral sequences $(x_1,x_2,\ldots,x_{2r+1})$ where $\sum_{i=1}^{2r+1}x_i=j-1$, with $x_i \geq 1$ for $1 \leq i \leq 2r$ and $x_{2r+1}\geq0$, and hence have cardinality given by $\binom{j-1}{2r}$; see, for example, \cite[p.\,15]{Stan}.  Note that if $\text{des}(\rho)=r$, then we must have $j\in[2r+1,n]$ with $0 \leq r \leq \lfloor(n-1)/2\rfloor$.  Summing over all possible $r$ and $j$ then gives
$$f_n(122)=\sum_{r=0}^{\lfloor\frac{n-1}{2}\rfloor}\sum_{j=2r+1}^n\binom{j-1}{2r}\binom{n-j+r}{r}=\sum_{r=0}^{\lfloor\frac{n-1}{2}\rfloor}\binom{n+r}{3r+1},$$
where we have made use of \cite[Identity~5.26]{GKP} in the second equality, which may readily be explained combinatorially (see \cite[Identity~136]{BQ}).

A similar argument applies to \eqref{Bxyc2e2}, except now when letters are added to $\rho \in \mathcal{F}_j(11)$ so as to generate $\pi \in \mathcal{F}_n(211)$, no runs of letters corresponding to descent bottoms of $\rho$ can have length greater than one in $\pi$.  Thus, there are now $j-r$ places in which to insert the $n-j$ letters into $\rho$ so as to produce $\pi$, and hence $\binom{(n-j)+j-r-1}{j-r-1}=\binom{n-r-1}{j-r-1}$ possible $\pi$ are produced for each $\rho \in \mathcal{F}_j(11)$ with $\text{des}(\rho)=r$.  Considering all possible $j$ and $r$ gives \eqref{Bxyc2e2}.

Finally, for \eqref{Bxyc2e3}, suppose $\pi \in \mathcal{F}_n(111)$ is formed by adding letters to $\rho \in \mathcal{F}_{n-k}(11)$.  Then, in order for $\pi$ to avoid 111, at most one extra letter of the same kind can be inserted so as to follow a letter in $\rho$, and hence there are $\binom{n-k}{k}$ possible $\pi$ generated from each $\rho$.  Note $n \geq 3$  and $0 \leq k \leq \lfloor n/2\rfloor$ implies $n-k\geq2$.  Thus, there are $\binom{n-k}{k}2^{n-k-2}$ members of $\mathcal{F}_n(111)$ whose skeleton has length $n-k$, by Lemma \ref{11avoidlem}. Summing over all $0 \leq k \leq \lfloor n/2\rfloor$ then implies \eqref{Bxyc2e3} and completes the proof of Corollary \ref{Bxyc2}. \hfill \qed

\section{Further distributions on flattened Catalan words}

In this section, we determine the joint distribution on $\mathcal{F}$ for three more triples of subword patterns, taken together with the length and $\text{trun}-1$.  We apply and extend the technique developed in the prior section in solving the functional equations which arise that are satisfied by these distributions.

To do so, we make use of generating functions of certain sequences related to arrays enumerating various subsets of $\mathcal{F}$.  Given an array $s=\left(s_{n,m}\right)_{1\leq m\leq n}$, define the sequences $u_n=u_n^{(s)}$, $v_n=v_n^{(s)}$ and $w_n=w_n^{(s)}$ by letting $u_n=\sum_{m=1}^n(m-1)s_{n,m}$, $v_n=\sum_{m=1}^ns_{n,m}$ and $w_n=s_{n,1}$ for $n \geq 1$.  We will omit the superscript in $u_n^{(s)}$ and the others when the array $s$ in question is understood.  Define the generating functions $U=U(x)$, $V=V(x)$ and $W=W(x)$ by letting $U=\sum_{n\geq1}u_nx^n$, $V=\sum_{n\geq1}v_nx^n$ and $W=\sum_{n\geq1}w_nx^n$.

\subsection{Distribution of $\#112$, $\#121$ and $\#221$}

Given $1 \leq m \leq n$, let $\mathcal{F}_{n,m}$ denote the subset of $\mathcal{F}_n$ consisting of those members $\pi$ where $\text{trun}(\pi)=m$. Let $c_{n,m}=c_{n,m}(p,q,r)$ be given by
$$c_{n,m}=\sum_{\pi \in \mathcal{F}_{n,m}}p^{\#112(\pi)}q^{\#121(\pi)}r^{\#221(\pi)}, \qquad 1 \leq m \leq n.$$
The array $c_{n,m}$ is given recursively as follows.

\begin{lemma}\label{112l1}
If $n \geq 3$, then
\begin{align}
c_{n,m}&=c_{n-1,m}+c_{n-1,m-1}+(p-1)c_{n-2,m-1}, \qquad 2 \leq m \leq n,\label{112l1e1}\\
c_{n,1}&=c_{n-1,1}+r\sum_{j=2}^{n-2}(j-1)c_{n-2,j}+\sum_{j=2}^{n-1}(j-2+q)(c_{n-2,j-1}+(p-1)c_{n-3,j-1}), \label{112l1e2}
\end{align}
with $c_{1,1}=c_{2,1}=c_{2,2}=1$ and $c_{0,m}=0$ for all $m \geq1$.
\end{lemma}
\begin{proof}
The initial conditions are easily verified, so assume $n \geq 3$.  To show \eqref{112l1e1}, first note that the weight of the members ending in a level and belonging to $\mathcal{F}_{n,m}$ where $m\geq 2$ is given by $c_{n-1,m}$.  So suppose that the last two letters of $\pi \in \mathcal{F}_{n,m}$ do not form a level, and hence must form an ascent as $m \geq 2$.  Such $\pi$ may be obtained by appending $a+1$ to any $\rho \in \mathcal{F}_{n-1,m-1}$ ending in $a$ for some $a \geq 1$.  If $\rho$ does not end in $a,a$, then there are $c_{n-1,m-1}-c_{n-2,m-1}$ possibilities for $\rho$, by subtraction, and hence the same number of possibilities for the corresponding $\pi$ obtained from $\rho$, as none of the three patterns in question are introduced in this case by the addition of the terminal $a+1$ letter. On the other hand, if $\rho$ does end in $a,a$, then appending $a+1$ introduces an occurrence of $112$, and thus there are $pc_{n-2,m-1}$ possibilities for $\pi$, where the factor of $p$ accounts for this additional 112.  Combining the prior cases above implies \eqref{112l1e1}.

For \eqref{112l1e2}, first note that there are $c_{n-1,1}$ possibilities for members of $\mathcal{F}_{n,1}$ whose last two letters are the same.  Otherwise, $m=1$ implies that the last two letters of $\pi \in \mathcal{F}_{n,1}$ form a descent.  So assume $\pi=\pi'\ell i$ for some $\ell \geq2$ and $i \in [\ell-1]$.  First suppose $\pi'$ has last letter $\ell$, with $\text{trun}(\pi')=j$ for some $j \in[\ell]$.  Note $i \leq \ell-j$ is not possible, as the sequence of descent bottoms (together with the first letter) of $\pi$ is nondecreasing.  So we have $i \in [\ell-j+1,\ell-1]$, and in particular, there are $j-1$ possibilities for $i$.  Since appending $i$ to $\pi'\ell$ in this case introduces a $221$, considering all possible $j$ implies that there are $r\sum_{j=2}^{n-2}(j-1)c_{n-2,j}$ possibilities for $\pi$.

So assume now $\pi'$ does not end in $\ell$, and hence it must end in $\ell-1$, for otherwise, $\pi$ would contain two consecutive descents, which would violate membership in $\mathcal{F}$. Thus, we have $\pi=\pi''(\ell-1)\ell i$ and suppose $\text{trun}(\pi''(\ell-1))=j-1$ for some $j \in [2,n-1]$.  Then there are $pc_{n-3,j-1}+(c_{n-2,j-1}-c_{n-3,j-1})$ possibilities concerning  $\pi''(\ell-1)\ell$ in this case if $n\geq4$, upon considering whether or not $\pi''$ ends in $\ell-1$.  The preceding expression is also seen to be correct (giving a value of $1$) if $n=3$, in which case $\pi''$ is empty and $\ell=2$.  For each choice of $\pi''(\ell-1)\ell$, there are $j-2+q$ possibilities for the appended letter $i$ in forming $\pi$.  To see this, note that we must have $i\in[\ell-j+1,\ell-1]$ and that the addition of $i=\ell-1$ introduces a 121, with none of the three subword patterns in question arising from the addition of $i<\ell-1$.  Thus, there are $(j-2+q)(c_{n-2,j-1}+(p-1)c_{n-3,j-1})$ possibilities for $\pi=\pi''(\ell-1)\ell i$ for each $j$. Considering all possible $j$ yields the second summation on the right side of \eqref{112l1e2} and completes the proof.
\end{proof}

Let $C(x,y)=\sum_{n\geq1}\sum_{m=1}^nc_{n,m}x^ny^{m-1}$ and we seek an explicit formula for $C(x,y)$.  Multiply both sides of \eqref{112l1e1} by $x^ny^{m-1}$ and sum over all $n \geq 3$ and $2 \leq m \leq n$ and then multiply both sides of \eqref{112l1e2} by $x^n$ and sum over $n \geq 3$.  Adding the two resulting equations gives
\begin{align*}
&C(x,y)-x-x^2(1+y)\\
&=x(C(x,y)-x)+xy(C(x,y)-x)+(p-1)x^2yC(x,y)+rx^2\frac{\partial}{\partial y}C(x,y)\mid_{y=1}\\
&\quad+\sum_{n\geq3}\sum_{j=2}^{n-1}(j-2)c_{n-2,j-1}x^n+(p-1)\sum_{n\geq4}\sum_{j=2}^{n-2}(j-2)c_{n-3,j-1}x^n+q\sum_{n\geq3}\sum_{j=2}^{n-1}c_{n-2,j-1}x^n\\
&\quad+(p-1)q\sum_{n\geq4}\sum_{j=2}^{n-2}c_{n-3,j-1}x^n\\
&=x(1+y+(p-1)xy)C(x,y)-x^2(1+y)+x^2(r+1+(p-1)x)\frac{\partial}{\partial y}C(x,y)\mid_{y=1}\\
&\quad+qx^2(1+(p-1)x)C(x,1),
\end{align*}
which may be rewritten to give the following result.

\begin{lemma}\label{112l2}
We have
\begin{equation}\label{112l2e1}
(1-x(1+y)-(p-1)x^2y)C(x,y)=x+x^2(r+1+(p-1)x)U+qx^2(1+(p-1)x)V,
\end{equation}
where $U=\frac{\partial}{\partial y}C(x,y)\mid_{y=1}$ and $V=C(x,1)$.
\end{lemma}

Solving the functional equation \eqref{112l2e1} yields the following result.

\begin{theorem}\label{Cxyth}
The generating function that enumerates the members of $\mathcal{F}$ jointly according to length, $\text{trun}-1$, $\#112$, $\#121$ and $\#221$  (marked by $x$, $y$, $p$, $q$ and $r$, respectively) is given by
\begin{equation}\label{Cxythe1}
C(x,y)=\frac{x(1-2x-(p-1)x^2)^2}{(1-x(1+y)-(p-1)x^2y)\alpha},
\end{equation}
where $$\alpha=1-4x-(2p+q-6)x^2+(4p+3q-r-pq-5)x^3+(p-1)(p+3q-r-3)x^4+(p-1)^2(q-1)x^5.$$
\end{theorem}
\begin{proof}
Consider the sequences $u_n$ and $v_n$ defined above, where here they are associated with the array $c_{n,m}$. Using \eqref{112l1e1} and \eqref{112l1e2}, one can establish the following system of recurrences for $n\geq3$:
\begin{align}
u_n&=2u_{n-1}+v_{n-1}+(p-1)(u_{n-2}+v_{n-2}),\label{Cxythe2}\\
v_n&=2v_{n-1}+(r+1)u_{n-2}+(p+q-1)v_{n-2}+(p-1)(u_{n-3}+qv_{n-3}),\label{Cxythe3}
\end{align}
with initial values $u_0=v_0=u_1=0$, $v_1=u_2=1$ and $v_2=2$.  By \eqref{Cxythe2} and \eqref{Cxythe3}, we have
\begin{align*}
U-x^2&=2xU+x(V-x)+(p-1)x^2(U+V),\\
V-x-2x^2&=2x(V-x)+(r+1)x^2U+(p+q-1)x^2V+(p-1)x^3(U+qV),
\end{align*}
and hence
\begin{align*}
U&=\frac{x(1+(p-1)x)}{1-2x-(p-1)x^2}V, \qquad \qquad V=\frac{x+x^2(r+1+(p-1)x)U}{1-2x-(p+q-1)x^2-(p-1)qx^3}.
\end{align*}
Solving the preceding system for $U$ and $V$ gives
$$U=\frac{x^2(1+(p-1)x)}{D}, \qquad \qquad V=\frac{x(1-2x-(p-1)x^2)}{D},$$
where $$D=(1-2x-(p-1)x^2)(1-2x-(p+q-1)x^2-(p-1)qx^3)-x^3(1+(p-1)x)(r+1+(p-1)x).$$
A computation shows $D=\alpha$, where $\alpha$ is as stated above.

By \eqref{112l2e1}, we then have
\begin{align}
(1-x(1+y)-(p-1)x^2y)C(x,y)&=x+\frac{x^3(1+(p-1)x)(q(1-2x-(p-1)x^2)+x(r+1+(p-1)x))}{\alpha}\notag\\
&=\frac{x\alpha+x^3\beta}{\alpha},\label{Cxythe4}
\end{align}
where $\beta=q+(pq+r-3q+1)x+(p-1)(r-3q+2)x^2-(p-1)^2(q-1)x^3$.
Formula \eqref{Cxythe1} now follows from \eqref{Cxythe4}, upon noting
$$x\alpha+x^3\beta=x-4x^2-2(p-3)x^3+4(p-1)x^4+(p-1)^2x^5=x(1-2x-(p-1)x^2)^2.$$
\end{proof}

Comparing the $q=r=1$ cases of $B(x,1)$ and $C(x,1)$, and also the $p=r=1$ case of $B(x,1)$ with the $p=q=1$ case of $C(x,1)$, yields the following equivalences.

\begin{corollary}\label{Cxythc1}
We have $112\approx122$ and $221\approx211$ as subwords on $\mathcal{F}_n$ for all $n \geq 1$.
\end{corollary}

Further, it is seen from the generating functions that both equivalences above respect trun.  Using the $p=r=1$ case of $C(x,1)$, one can obtain the following formulas.

\begin{corollary}\label{Cxythc2}
We have
$$\text{tot}_n(121)=\frac{(n+1)3^{n-3}+n-3}{4}, \qquad n \geq 2,$$
and
$$f_n(121)=\sum_{j=1}^n\sum_{p=0}^{\lfloor \frac{j-1}{2}\rfloor}\sum_{r=0}^p\binom{p}{r}\binom{j-p-1}{2p-r}\binom{n-r-1}{j-1}, \qquad n \geq 1.$$
\end{corollary}
\noindent   A combinatorial proof of the last two corollaries, along with some further results, is given in the final subsection.

By setting one of $\{p,q,r\}$ equal to zero, we obtain the joint distribution of two of the corresponding subwords on the subset of $\mathcal{F}_n$ whose members avoid the other pattern.  For example, setting $p=0$ in \eqref{Cxythe1} implies that the generating function for the joint distribution of the 121 and 221 subwords on the set $\mathcal{F}_n(112)$ for $n\geq 1$ is given by
$$C(x,1;0,p,q)=\frac{x(1-x)}{1-3x-(p-3)x^2+(2p-q-2)x^3-(p-1)x^4}.$$
Setting $p$ or $q$ equal to unity, one obtains the generating functions for the corresponding univariate distributions of 221 and 121 respectively on $\mathcal{F}_n(112)$.  Similar remarks apply to the other joint distributions found in this paper.  The problem of considering occurrences of one pattern while avoiding another is one that has been considered previously on other discrete structures, especially permutations (see, e.g., \cite{TM,Rob}).

\subsection{Distribution of $\#123$, $\#231$ and $\#221$}\label{ssec3.2}

Let $d_{n,m}=d_{n,m}(p,q,r)$ be given by
$$d_{n,m}=\sum_{\pi \in \mathcal{F}_{n,m}}p^{\#123(\pi)}q^{\#231(\pi)}r^{\#221(\pi)}, \qquad 1 \leq m \leq n.$$
The array $d_{n,m}$ is given recursively as follows.

\begin{lemma}\label{123l1}
If $n \geq 3$, then
\begin{align}
d_{n,m}&=d_{n-1,m}+pd_{n-1,m-1}+(1-p)d_{n-2,m-1}, \qquad 3 \leq m \leq n,\label{123l1e1}\\
d_{n,2}&=d_{n-1,2}+d_{n-1,1},\label{123l1e2}\\
d_{n,1}&=d_{n-1,1}+(1-p)(d_{n-2,1}-d_{n-3,1})+(pq+r)\sum_{j=2}^{n-2}(j-1)d_{n-2,j}+(1-p)q\sum_{j=2}^{n-3}(j-1)d_{n-3,j}\notag\\
       &\quad+\sum_{j=1}^{n-2}(pd_{n-2,j}+(1-p)d_{n-3,j}),\label{123l1e3}
\end{align}
with $d_{1,1}=d_{2,1}=d_{2,2}=1$ and $d_{0,m}=0$ for all $m \geq1$.
\end{lemma}
\begin{proof}
Suppose $\rho \in \mathcal{F}_{n,m}$ ends in $a,b$, where $n \geq 3$ and $m \in [n]$.  If $a=b$, then there are $d_{n-1,m}$ possibilities for $\rho$, regardless of $m$. So assume henceforth $a \neq b$.  If $m \geq 2$, then $\rho$ cannot end in a descent, and thus $b=a+1$. If $m=2$, then there are $d_{n-1,1}$ members of $\mathcal{F}_{n,2}$ ending in $a,a+1$ for some $a\geq 1$, as appending $a+1$ to a member of $\mathcal{F}_{n-1,1}$ ending in $a$ is seen not to introduce an occurrence of one of $\{123,231,221\}$, while at the same time increasing trun from one to two.  Combining with the prior case above yields \eqref{123l1e2}.  On other hand, if $m \geq 3$, first note that there are $d_{n-2,m-1}$ possibilities for $\rho$ of the form $\rho=\rho'aa(a+1)$.  Further, the members of $\mathcal{F}_{n-1,m-1}$ whose last two letters form an ascent have weight $d_{n-1,m-1}-d_{n-2,m-1}$, by subtraction as $m\geq 3$, upon excluding members ending in a level.  This implies that the weight of $\rho \in \mathcal{F}_{n,m}$ of the form $\rho=\rho'(a-1)a(a+1)$ is given by $p(d_{n-1,m-1}-d_{n-2,m-1})$, where the extra factor of $p$ accounts for the 123 witnessed by the final three letters of $\rho$.  Combining with the previous cases yields \eqref{123l1e1}.

For \eqref{123l1e3}, we may assume $n \geq 4$, as both sides of \eqref{123l1e3} are seen to equal two when $n=3$.  Note that a member of $\mathcal{F}_{n,1}$ whose last two letters are not equal must end in a descent.  Suppose $\pi=\pi'x\ell i \in \mathcal{F}_{n,1}$, where $n \geq 4$, $\ell \geq 2$, $i \in [\ell-1]$ and $x$ denotes the antepenultimate letter of $\pi$.  Note $\pi \in \mathcal{F}$ implies $x\in\{\ell-1,\ell\}$.  If $x=\ell$, then we get a weight of $r\sum_{j=2}^{n-2}(j-1)d_{n-2,j}$ for such $\pi$, upon considering all possible values of $\text{trun}(\pi'\ell)$, where the factor of $r$ accounts for the occurrence of 221 arising when $\ell,i$ is appended to $\pi'\ell$ in forming $\pi$.  So assume $x=\ell-1$ and we consider cases on $u:=\text{trun}(\pi'(\ell-1))$.  If $u=1$, then there are $d_{n-2,1}$ possibilities for $\pi'(\ell-1)$, and hence for $\pi=\pi'(\ell-1)\ell i$ as well, since $u=1$ implies $i=\ell-1$ and thus none of the three patterns are introduced when $\ell,\ell-1$ is appended to $\pi'(\ell-1)$.

So assume $u \geq 2$, which implies $\ell \geq 3$. Let $u=j-1$ for some $j\in[3,n-1]$.  By prior arguments, there are $p(d_{n-2,j-1}-d_{n-3,j-1})+d_{n-3,j-1}$  possibilities for $\pi'(\ell-1)\ell$, as appending $\ell$ to $\pi'(\ell-1)$ when $\pi'$ ends in $\ell-2$ introduces a 123 and accounts for the extra factor of $p$. Once $\pi'(\ell-1)\ell$ is formed, we then append $i \in [\ell-j+1,\ell-1]$ so as to obtain $\pi$.  If $i=\ell-1$, then it is seen that no occurrences of any of the three patterns arise when $i$ is appended, whereas if $i \in [\ell-j+1,\ell-2]$, then a 231 arises.  Thus, the letter $i$ is accounted for by $q(j-2)+1$, whence there are $(q(j-2)+1)(pd_{n-2,j-1}+(1-p)d_{n-3,j-1})$ possibilities for $\pi=\pi'(\ell-1)\ell i$ when $\text{trun}(\pi'(\ell-1))=j-1$.  Considering all possible $j$ then yields a weight of
$$\sum_{j=3}^{n-1}(q(j-2)+1)(pd_{n-2,j-1}+(1-p)d_{n-3,j-1})$$
for $\pi \in \mathcal{F}_{n,1}$ of the form $\pi'(\ell-1)\ell i$ where $i<\ell$ and $\text{trun}(\pi'(\ell-1))\geq 2$.  Combining with the prior cases, we get
\begin{align*}
d_{n,1}&=d_{n-1,1}+d_{n-2,1}+r\sum_{j=2}^{n-2}(j-1)d_{n-2,j}+\sum_{j=3}^{n-1}(q(j-2)+1)(pd_{n-2,j-1}+(1-p)d_{n-3,j-1})\\
&=d_{n-1,1}+d_{n-2,1}+r\sum_{j=2}^{n-2}(j-1)d_{n-2,j}+pq\sum_{j=2}^{n-2}(j-1)d_{n-2,j}+(1-p)q\sum_{j=2}^{n-3}(j-1)d_{n-3,j}\\
&\quad+p\sum_{j=2}^{n-2}d_{n-2,j}+(1-p)\sum_{j=2}^{n-3}d_{n-3,j},
\end{align*}
which may be rewritten to give \eqref{123l1e3}.
\end{proof}

Let $D(x,y)=\sum_{n\geq1}\sum_{m=1}^nd_{n,m}x^ny^{m-1}$.  Multiply both sides of \eqref{123l1e1} by $x^ny^{m-1}$ and sum over all $n \geq 3$ and $3 \leq m \leq n$.  Then multiply both sides of \eqref{123l1e2} and \eqref{123l1e3} by $x^ny$ and $x^n$, respectively, and sum over $n \geq 3$.  Adding the three resulting equations, we get
\begin{align*}
&D(x,y)-x-x^2(1+y)\\
&=\sum_{n\geq3}\sum_{m=1}^{n-1}d_{n-1,m}x^ny^{m-1}+p\sum_{n\geq3}\sum_{m=2}^{n}d_{n-1,m-1}x^ny^{m-1}+(1-p)\sum_{n\geq3}\sum_{m=2}^{n-1}d_{n-2,m-1}x^ny^{m-1}\\
&\quad+(1-p)y\left(\sum_{n\geq3}d_{n-1,1}x^n-\sum_{n\geq3}d_{n-2,1}x^n\right)+(1-p)\left(\sum_{n\geq3}d_{n-2,1}x^n-\sum_{n\geq4}d_{n-3,1}x^n\right)\\
&\quad+(pq+r)\sum_{n\geq4}\sum_{j=2}^{n-2}(j-1)d_{n-2,j}x^n+(1-p)q\sum_{n\geq5}\sum_{j=2}^{n-3}(j-1)d_{n-3,j}x^n+p\sum_{n\geq3}\sum_{j=1}^{n-2}d_{n-2,j}x^n\\
&\quad+(1-p)\sum_{n\geq4}\sum_{j=1}^{n-3}d_{n-3,j}x^n\\
&=x(1+py+(1-p)xy)D(x,y)-x^2(1+py)+(1-p)x(1-x)yD(x,0)-(1-p)x^2y\\
&\quad+(1-p)x^2(1-x)D(x,0)+x^2(pq+r+(1-p)qx)\frac{\partial}{\partial y}D(x,y)\mid_{y=1}+x^2(p+(1-p)x)D(x,1),
\end{align*}
which may be rewritten to give the following result.

\begin{lemma}\label{123l2}
We have
\begin{align}
(1-x(1+py)-(1-p)x^2y)D(x,y)&=x+x^2(pq+r+(1-p)qx)U+x^2(p+(1-p)x)V\notag\\
&\quad+(1-p)x(1-x)(x+y)W,\label{123l2e1}
\end{align}
where $U=\frac{\partial}{\partial y}D(x,y)\mid_{y=1}$, $V=D(x,1)$ and $W=D(x,0)$.
\end{lemma}

To determine $D(x,y)$, we seek explicit formulas for the generating functions $U$, $V$ and $W$ whose coefficients satisfy the following system of recurrences.

\begin{lemma}\label{123l3}
If $n \geq3$, then
\begin{align}
u_n&=(1+p)u_{n-1}+pv_{n-1}+(1-p)(u_{n-2}+v_{n-2}+w_{n-1}-w_{n-2}),\label{123l3e1}\\
v_n&=(1+p)v_{n-1}+w_n-pw_{n-1}+(1-p)(v_{n-2}-w_{n-2}),\label{123l3e2}\\
w_n&=(pq+r)u_{n-2}+(1-p)qu_{n-3}+pv_{n-2}+w_{n-1}+(1-p)(v_{n-3}+w_{n-2}-w_{n-3}),\label{123l3e3}
\end{align}
with initial values $u_0=v_0=w_0=u_1=0$, $v_1=w_1=u_2=w_2=1$ and $v_2=2$.
\end{lemma}
\begin{proof}
The initial conditions for $0 \leq n \leq 2$ are clear, so assume $n \geq 3$.  By \eqref{123l1e1} and \eqref{123l1e2}, we have
\begin{align*}
u_n&=\sum_{m=2}^n(m-1)d_{n,m}\\
&=d_{n-1,1}+\sum_{m=2}^{n-1}(m-1)d_{n-1,m}+p\sum_{m=3}^n(m-1)d_{n-1,m-1}+(1-p)\sum_{m=3}^{n-1}(m-1)d_{n-2,m-1}\\
&=w_{n-1}+u_{n-1}+p\sum_{m=2}^{n-1}(m-1)d_{n-1,m}+p\sum_{m=2}^{n-1}d_{n-1,m}+(1-p)\sum_{m=2}^{n-2}(m-1)d_{n-2,m}+(1-p)\sum_{m=2}^{n-2}d_{n-2,m}\\
&=w_{n-1}+u_{n-1}+p(u_{n-1}+v_{n-1}-w_{n-1})+(1-p)(u_{n-2}+v_{n-2}-w_{n-2}),
\end{align*}
which implies \eqref{123l3e1}.  Rewriting \eqref{123l1e3} in terms of $u_n$, $v_n$ and $w_n$ yields
$$w_n=w_{n-1}+(1-p)(w_{n-2}-w_{n-3})+(pq+r)u_{n-2}+(1-p)qu_{n-3}+pv_{n-2}+(1-p)v_{n-3},$$
which gives \eqref{123l3e3}.  Finally, summing \eqref{123l1e1} over $3 \leq m \leq n$ for a fixed $n\geq 3$ and adding the resulting equation to the sum of \eqref{123l1e2} and \eqref{123l1e3}, we have
\begin{align}
v_n&=\sum_{m=1}^nd_{n,m}\notag\\
&=v_{n-1}+p(v_{n-1}-w_{n-1})+(1-p)(v_{n-2}-w_{n-2})+w_{n-1}+(1-p)(w_{n-2}-w_{n-3})+(pq+r)u_{n-2}\notag\\
&\quad+(1-p)qu_{n-3}+pv_{n-2}+(1-p)v_{n-3}\notag\\
&=(pq+r)u_{n-2}+(1-p)qu_{n-3}+(1+p)v_{n-1}+v_{n-2}+(1-p)(v_{n-3}+w_{n-1}-w_{n-3}).\label{123l3e4}
\end{align}
Subtracting \eqref{123l3e3} from \eqref{123l3e4} leads to \eqref{123l3e2} and completes the proof.
\end{proof}

\begin{theorem}\label{Dxyth}
The generating function that enumerates the members of $\mathcal{F}$ jointly according to length, $\text{trun}-1$, $\#123$, $\#231$ and $\#221$  (marked by $x$, $y$, $p$, $q$ and $r$, respectively) is given by
\begin{align}
D(x,y)&=\frac{x}{1-x(1+py)+(p-1)x^2y}\notag\\
&\quad+\frac{x^2+(pq-p+r-1)x^3-(p-1)(q-1)x^4-xy(p-1)(1-(p+1)x+(p-1)x^2)^2}{(1-(p-1)x)(1-(p+1)x+(p-1)x^2)(1-x(1+py)+(p-1)x^2y)}V,\label{Dxythe1}
\end{align}
where $$V=D(x,1)=\frac{x(1-(p-1)x)(1-(p+1)x+(p-1)x^2)}{1-2(p+1)x+(p^2+4p-2)x^2-(2p^2+pq-p+r-3)x^3+(p-1)(p+q-2)x^4}.$$
\end{theorem}
\begin{proof}
Rewriting \eqref{123l3e1}--\eqref{123l3e3} in terms of generating functions, we have
\begin{align}
(1-(p+1)x+(p-1)x^2)U&=x(p-(p-1)x)V-(p-1)x(1-x)W, \label{Dxythe2}\\
W&=\frac{1-(p+1)x+(p-1)x^2}{(1-x)(1-(p-1)x)}V,\label{Dxythe3}\\
(1-x)(1+(p-1)x^2)W&=x+x^2(pq+r-(p-1)qx)U+x^2(p-(p-1)x)V. \label{Dxythe4}
\end{align}
Substituting \eqref{Dxythe3} into \eqref{Dxythe2}, and solving for $U$ in terms of $V$, gives
\begin{equation}\label{Dxythe5}
U=\frac{x}{(1-(p-1)x)(1-(p+1)x+(p-1)x^2)}V.
\end{equation}
Substituting \eqref{Dxythe3} and \eqref{Dxythe5} into \eqref{Dxythe4} implies
\begin{align*}
\frac{(1+(p-1)x^2)(1-(p+1)x+(p-1)x^2)}{1-(p-1)x}V&=x+x^2(p-(p-1)x)V\\
&\quad+\frac{x^3(pq+r-(p-1)qx)}{(1-(p-1)x)(1-(p+1)x+(p-1)x^2)}V,
\end{align*}
and solving for $V$ yields its stated formula above.  By \eqref{123l2e1}, \eqref{Dxythe3} and \eqref{Dxythe5}, we have
\begin{align*}
&(1-x(1+py)+(p-1)x^2y)D(x,y)\\
&=x+x^2(p-(p-1)x)V-\frac{(p-1)x(x+y)(1-(p+1)x+(p-1)x^2)}{1-(p-1)x}V\\
&\quad+\frac{x^3(pq+r-(p-1)qx)}{(1-(p-1)x)(1-(p+1)x+(p-1)x^2)}V\\
&=x+\frac{x^2-xy(p-1)(1-(p+1)x+(p-1)x^2)}{1-(p-1)x}V+\frac{x^3(pq+r-(p-1)qx)}{(1-(p-1)x)(1-(p+1)x+(p-1)x^2)}V.
\end{align*}
Combining the last two terms containing $V$, and solving for $D(x,y)$, yields \eqref{Dxythe1}.
\end{proof}

Note that the $p=q=1$ cases of $C(x,y)$ and $D(x,y)$ are equal, as they should be, and give the generating function for the joint distribution of trun and $221$ on $\mathcal{F}_n$ for $n\geq 1$.  We have the following explicit formulas for $\text{tot}_n(123)$ and $f_n(123)$, which may be obtained from the $q=r=y=1$ case of Theorem \ref{Dxyth}.

\begin{corollary}\label{Dxyc1}
We have $\text{tot}_n(123)=(n-2)3^{n-3}$ for $n \geq 2$
and
$$f_n(123)=1+\sum_{j=2}^n\sum_{\ell=1}^{j-1}\binom{j-2}{\ell-1}\binom{n-j+\ell}{j-1}, \qquad n \geq 1.$$
\end{corollary}

  From Theorem \ref{Dxyth}, we also have
$$D(x,1;1,p,q)=\frac{x(1-2x)}{1-4x+3x^2-(p+q-2)x^3}=D(x,1;1,q,p),$$
and hence the joint distribution
$$d_n(p,q):=\sum_{\pi \in \mathcal{F}_n}p^{\#231(\pi)}q^{\#221(\pi)}, \qquad n \geq 1,$$
is symmetric in $p$ and $q$.  In particular, the statistics $\#231$ and $\#221$ are identically distributed on $\mathcal{F}_n$, and hence $\text{tot}_n(231)=\text{tot}_n(221)$ and $f_n(231)=f_n(221)$ for all $n$.  To realize $d_n(p,q)=d_n(q,p)$ bijectively, consider the involution on $\mathcal{F}_n$ that replaces each string of the form $a(a+1)b$, where $a>b$, with $aab$, and vice versa.

\subsection{Distribution of $\#112$, $\#212$ and $\#312$}\label{ssec3.3}

Let $e_{n,m}=e_{n,m}(p,q,r)$ be given by
$$e_{n,m}=\sum_{\pi \in \mathcal{F}_{n,m}}p^{\#112(\pi)}q^{\#212(\pi)}r^{\#312(\pi)}, \qquad 1 \leq m \leq n.$$
The $e_{n,m}$ satisfy the following recursion.

\begin{lemma}\label{112al1}
If $n \geq 3$, then
\begin{align}
e_{n,m}&=e_{n-1,m}+e_{n-1,m-1}+(p-1)e_{n-2,m-1}, \qquad 3 \leq m \leq n, \label{112al1e1}\\
e_{n,2}&=e_{n-1,2}+pe_{n-2,1}+\sum_{j=2}^{n-2}(r(j-2)+q)e_{n-2,j},\label{112al1e2}\\
e_{n,1}&=e_{n-1,1}+\sum_{j=2}^{n-1}(j-1)e_{n-1,j},\label{112al1e3}
\end{align}
with $e_{1,1}=e_{2,1}=e_{2,2}=1$.
\end{lemma}
\begin{proof}
First, note that the weight of all members of $\mathcal{F}_{n,m}$ whose last two letters are the same where $n \geq 3$ and $m \in [n]$ is given by $e_{n-1,m}$ and that the weight of those ending in $a,a,a+1$ for some $a \geq 1$ is $pe_{n-2,m-1}$ if $m\geq2$. For $m\geq3$, the only other possibility is ending in $a,a+1,a+2$, and for this, it is seen that such members of $\mathcal{F}_{n,m}$ have weight $e_{n-1,m-1}-e_{n-2,m-1}$, by subtraction.  Combining with the previous cases implies \eqref{112al1e1}.  Note that a member of $\mathcal{F}_{n,1}$ not ending in a level must end in a descent.  Suppose $\rho=\rho'a \in \mathcal{F}_{n,1}$, where the last letter of $\rho'$ exceeds $a$.  Considering all possible values of $\text{trun}(\rho')$ accounts for the summation on the right side of \eqref{112al1e3}. Combining with the prior case of ending in a level then yields \eqref{112al1e3}.

On the other hand, if $m=2$, then we must account for $\pi \in \mathcal{F}_{n,2}$ where $n\geq4$ that are expressible as $\pi=\pi'\ell a(a+1)$ such that $\ell \geq 2$ and $a \in [\ell-1]$.  Suppose $j=\text{trun}(\pi'\ell)$, where $j \in [2,n-2]$. Then $\ell \geq j$ and $a \in [\ell-j+1,\ell-1]$ since $a$ is a descent bottom.  If $a=\ell-1$, then appending $\ell-1,\ell$ to $\pi'\ell$ to form $\pi$ introduces an occurrences of $212$, whereas if $a \in [\ell-j+1,\ell-2]$, then appending $a,a+1$ gives rise to a $312$.  Thus, the final two letters of $\pi$ are accounted for by $r(j-2)+q$.  As $\pi'\ell$ is enumerated by $e_{n-2,j}$, we get $(r(j-2)+q)e_{n-2,j}$ possibilities for $\pi$ such that $\text{trun}(\pi'\ell)=j$.  Summing over all $j \in [2,n-2]$ then gives $\sum_{j=2}^{n-2}(r(j-2)+q)e_{n-2,j}$ possible $\pi \in \mathcal{F}_{n,2}$ ending in $\ell,a,a+1$ for some $\ell>a\geq 1$.  Combining with the prior cases yields \eqref{112al1e2} and completes the proof.
\end{proof}

Let $E(x,y)=\sum_{n\geq1}\sum_{m=1}^ne_{n,m}x^ny^{m-1}$. By \eqref{112al1e1}--\eqref{112al1e3}, we have
\begin{align*}
&E(x,y)-x-x^2(1+y)\\
&=x(E(x,y)-x)+\sum_{n\geq3}\sum_{m=3}^ne_{n-1,m-1}x^ny^{m-1}+p\sum_{n\geq3}\sum_{m=2}^{n-1}e_{n-2,m-1}x^ny^{m-1}-\sum_{n\geq4}\sum_{m=3}^{n-1}e_{n-2,m-1}x^ny^{m-1}\\
&\quad+ry\sum_{n\geq4}\sum_{j=2}^{n-2}(j-1)e_{n-2,j}x^n+(q-r)y\sum_{n\geq4}\sum_{j=2}^{n-2}e_{n-2,j}x^n+\sum_{n\geq3}\sum_{j=2}^{n-1}(j-1)e_{n-1,j}x^n\\
&=x(E(x,y)-x)+xy(E(x,y)-E(x,0))+px^2yE(x,y)-x^2y(E(x,y)-E(x,0))\\
&\quad+rx^2y\frac{\partial}{\partial y}E(x,y)\mid_{y=1}+(q-r)x^2y(E(x,1)-E(x,0))+x\frac{\partial}{\partial y}E(x,y)\mid_{y=1},
\end{align*}
which may be rewritten to give the following functional equation for $E(x,y)$.

\begin{lemma}\label{112al2}
We have
\begin{equation}\label{112al2e1}
(1-x-xy-(p-1)x^2y)E(x,y)=x+x^2y+x(1+rxy)U+(q-r)x^2yV-xy(1-(r-q+1)x)W,
\end{equation}
where $U=\frac{\partial}{\partial y}E(x,y)\mid_{y=1}$, $V=E(x,1)$ and $W=E(x,0)$.
\end{lemma}

\begin{lemma}\label{112al3}
If $n \geq3$, then
\begin{align}
u_n&=u_{n-1}+(p-1)u_{n-2}+v_n-v_{n-1},\label{112al3e1}\\
v_n&=u_{n-1}+ru_{n-2}+2v_{n-1}+(p+q-r-1)v_{n-2}-w_{n-1}+(r-q+1)w_{n-2},\label{112al3e2}\\
w_n&=u_{n-1}+w_{n-1},\label{112al3e3}
\end{align}
with initial values $u_0=v_0=w_0=u_1=0$, $v_1=w_1=u_2=w_2=1$ and $v_2=2$.
\end{lemma}
\begin{proof}
Formula \eqref{112al3e3} follows immediately from \eqref{112al1e3} and the definitions of the various sequences.  By \eqref{112al1e1}--\eqref{112al1e3}, we have
\begin{align*}
v_n&=\sum_{m=1}^ne_{n,m}\\
&=v_{n-1}+\sum_{m=3}^ne_{n-1,m-1}+(p-1)\sum_{m=2}^{n-1}e_{n-2,m-1}+e_{n-2,1}+r\sum_{j=2}^{n-2}(j-1)e_{n-2,j}+(q-r)\sum_{j=2}^{n-2}e_{n-2,j}+u_{n-1}\\
&=2v_{n-1}-w_{n-1}+(p-1)v_{n-2}+w_{n-2}+ru_{n-2}+(q-r)(v_{n-2}-w_{n-2})+u_{n-1},
\end{align*}
which implies \eqref{112al3e2}.  Using \eqref{112al1e1} and \eqref{112al1e2}, one can show
\begin{align}
u_n&=\sum_{m=2}^n(m-1)e_{n,m}\notag\\
&=u_{n-1}+\sum_{m=3}^n(m-1)e_{n-1,m-1}+(p-1)\sum_{m=2}^{n-1}(m-1)e_{n-2,m-1}+e_{n-2,1}+r\sum_{j=2}^{n-2}(j-1)e_{n-2,j}\notag\\
&\quad+(q-r)\sum_{j=2}^{n-2}e_{n-2,j}\notag\\
&=2u_{n-1}+v_{n-1}-w_{n-1}+(p-1)(u_{n-2}+v_{n-2})+w_{n-2}+ru_{n-2}+(q-r)(v_{n-2}-w_{n-2})\notag\\
&=2u_{n-1}+(p+r-1)u_{n-2}+v_{n-1}+(p+q-r-1)v_{n-2}-w_{n-1}+(r-q+1)w_{n-2}.\label{112al3e4}
\end{align}
Subtracting \eqref{112al3e2} from \eqref{112al3e4}  yields \eqref{112al3e1} and completes the proof.
\end{proof}

\begin{theorem}\label{Exyth}
The generating function that enumerates the members of $\mathcal{F}$ jointly according to length, $\text{trun}-1$, $\#112$, $\#212$ and $\#312$  (marked by $x$, $y$, $p$, $q$ and $r$, respectively) is given by
\begin{align}
E(x,y)&=\frac{x-3x^2-(p-3)x^3+(p-1)x^4-x^3y(q-1+(r-2q+1)x+(p-1)(r-q)x^2)}{(1-x)(1-x-(p-1)x^2)(1-x-xy-(p-1)x^2y)}\notag\\
&\quad+\frac{x-x^2+x^2y(q-1+(r-2q+1)x+(p-1)(r-q)x^2)}{(1-x-(p-1)x^2)(1-x-xy-(p-1)x^2y)}V,\label{Exythe1}
\end{align}
where $V=E(x,1)=\frac{x\alpha}{(1-x)\beta}$ with
\begin{align*}
\alpha&=1-3x-(p+q-4)x^2+(p+2q-r-2)x^3+(p-1)(q-r)x^4,\\
\beta&=1-4x-(2p+q-6)x^2+(3p+2q-r-4)x^3+(p-1)(p+q-r-1)x^4.
\end{align*}
\end{theorem}
\begin{proof}
Rewriting \eqref{112al3e1}--\eqref{112al3e3} in terms of generating functions yields the system
\begin{align}
(1-x-(p-1)x^2)U&=(1-x)V-x,\label{Exythe2}\\
(1-2x-(p+q-r-1)x^2)V&=x+x^2+x(1+rx)U-x(1-(r-q+1)x)W,\label{Exythe3}\\
W&=\frac{x(1+U)}{1-x}. \label{Exythe4}
\end{align}
Substituting \eqref{Exythe4} into \eqref{Exythe3}, and simplifying, gives
$$(1-2x-(p+q-r-1)x^2)V=\frac{x(1-x+(r-q)x^2)}{1-x}+\frac{x(1+(r-2)x-(q-1)x^2)}{1-x}U.$$
By \eqref{Exythe2}, we then have
\begin{align*}
(1-2x-(p+q-r-1)x^2)V&=\frac{x(1-x+(r-q)x^2)}{1-x}+\frac{x(1+(r-2)x-(q-1)x^2)}{1-x-(p-1)x^2}V\\
&\quad-\frac{x^2(1+(r-2)x-(q-1)x^2)}{(1-x)(1-x-(p-1)x^2)},
\end{align*}
and solving for $V$ gives its stated formula above.

Rewriting \eqref{112al2e1} in terms of $U$ using \eqref{Exythe2} and \eqref{Exythe4} implies
\begin{align*}
&(1-x-xy-(p-1)x^2y)E(x,y)\\
&=x+x^2y+x(1+rxy)U+(q-r)x^2y\left(\frac{x}{1-x}+\frac{1-x-(p-1)x^2}{1-x}U\right)\\
&\quad-xy(1-(r-q+1)x)\left(\frac{x}{1-x}+\frac{x}{1-x}U\right)\\
&=x+\left(x+(r-1)x^2y+\frac{(q-r)x^2y(1-2x-(p-1)x^2)}{1-x}\right)U\\
&=x+\frac{x-x^2+x^2y(q-1+(r-2q+1)x+(p-1)(r-q)x^2)}{1-x}U.
\end{align*}
Thus, by \eqref{Exythe2}, we have
\begin{align*}
&(1-x-xy-(p-1)x^2y)E(x,y)\\
&=x+\frac{x-x^2+x^2y(q-1+(r-2q+1)x+(p-1)(r-q)x^2)}{1-x-(p-1)x^2}V\\
&\quad-\frac{x^2-x^3+x^3y(q-1+(r-2q+1)x+(p-1)(r-q)x^2)}{(1-x)(1-x-(p-1)x^2)},
\end{align*}
which leads to \eqref{Exythe1} after some algebra and completes the proof.
\end{proof}

Note that the $q=r=1$ case of $E(x,y)$ is seen to equal the corresponding case of $C(x,y)$, with both enumerating members of $\mathcal{F}_n$ jointly according to trun and $\#112$. From Theorem \ref{Exyth}, we obtain the following subword totals and avoidance formulas on $\mathcal{F}_n$ for the patterns 212 and 312.

\begin{corollary}\label{Exyc1}
If $n \geq 2$, then $$\text{tot}_n(212)=\frac{(n-2)(3^{n-3}-1)}{4}, \qquad \text{tot}_n(312)=\frac{(n-5)3^{n-3}+n-1}{4}.$$
If $n \geq 1$, then
$$f_n(212)=1+\sum_{j=2}^n\sum_{p=1}^{\lfloor j/2\rfloor}\sum_{r=0}^{p-1}\binom{p-1}{r}\binom{j-p}{p+r}\binom{n-p+r}{j-1}$$
and
$$f_n(312)=1+\sum_{j=2}^n\sum_{p=1}^{\lfloor j/2\rfloor}\sum_{r=0}^{p-1}\binom{p-1}{r}\binom{j-p}{p+r}\binom{n-r-1}{j-1}.$$
\end{corollary}
\medskip
\noindent{\bf Remarks:}  We have $\text{tot}_n(123)=A027471[n-1]$ for $n\geq 2$ and $\text{tot}_n(312)=A212337[n-5]$ for $n \geq 5$.  In \cite{BHR}, a \emph{valley} within a Catalan word is defined as a string of the form $ab^\ell(b+1)$ for some $\ell \geq 1$, where $a>b$.  A comparison with \cite[Corollary~4.7]{BHR} shows that $\text{tot}_n(312)$ equals the total number of valleys in all the members of $\mathcal{F}_{n-1}$.  This fact may be realized bijectively as follows.  Suppose $ab(b+1)$ is a marked occurrence of 312 in some member of $\mathcal{F}_n$, where $a>b+1$.  Let $\ell$ denote the length of the run of the letter $a$. That is, we have the string $(a-1)a^\ell b(b+1)$, with the final three letters comprising the marked occurrence of 312.  We then replace this string with $(a-1)b^{\ell}(b+1)$, which is then marked.  This then yields all marked members of $\mathcal{F}_{n-1}$ wherein a valley per the definition above is marked.

Furthermore, comparing Corollaries \ref{Bxyc1} and \ref{Exyc1}, we have that $\text{tot}_n(211)=\text{tot}_n(212)$ for all $n$ despite 211 and 212 having different distributions on $\mathcal{F}_n$, and it is possible to provide a bijective proof of this fact which leverages the bijection used to establish Lemma \ref{11avoidlem} whose details we omit. Finally, a \emph{short} valley refers to one in which $\ell=1$ in the definition above.  Note that the number of short valleys within a Catalan word equals the sum of the number of occurrences of 212 and 312.  Taking $p=y=1$ and $r=q$ in Theorem \ref{Exyth} implies that the generating function enumerating the members of $\mathcal{F}_n$ for $n \geq 1$ according to the number of short valleys is given by
$$E(x,1;1,q,q)=\frac{x(1-2x-(q-1)x^2)}{(1-x)(1-3x-(q-1)x^2)}.$$
This is in accordance with the formula found in \cite[Theorem~4.1]{BHR}.

\subsection{Combinatorial proofs}

Here, we provide direct counting arguments of the formulas for $\text{tot}_n(\tau)$ for the three-letter subword patterns $\tau$ which do not follow from Corollary \ref{Axyc1}.  We also prove the avoidance results from this section as well as the  equivalences in Corollary \ref{Cxythc1}. \medskip

\noindent\emph{Proofs of formulas for $\text{tot}_n(\tau)$ for $\tau\in\{212,121,312,123\}$:} \medskip

We first make some preliminary observations that apply to all four of the patterns in question. Given a subword $\tau$, let $\mathcal{G}_n(\tau)$ denote the set obtained from members of $\mathcal{F}_n$ by marking a single occurrence of $\tau$, if it exists.  Note that members of $\mathcal{F}_n$ that avoid $\tau$ do not give rise to any members of $\mathcal{G}_n(\tau)$.  Then $\text{tot}_n(\tau)=|\mathcal{G}_n(\tau)|$ for all $n \geq 1$, and we seek to enumerate members of $\mathcal{G}_n(\tau)$.  To do so, let $\mathcal{G}_n^*(\tau)$ denote the subset of $\mathcal{G}_n(\tau)$ whose members contain no levels.

Consider forming members of $\mathcal{G}_n(\tau)$ from skeletons in $\mathcal{G}_j^*(\tau)$ for some $j \in [n]$ in a manner analogous to the one described above wherein a total of $n-j$ extra copies of letters are inserted directly following the corresponding entries in a skeleton. Suppose $\tau$ itself does not contain any levels and consider a skeleton $\rho\in \mathcal{G}_j^*(\tau)$, where $j\geq |\tau|\geq 3$. Let $s$ denote the string of letters in $\rho$ corresponding to the marked occurrence of $\tau$. Then letters may be inserted into $\rho$ that directly follow any of the entries in $\rho$ except for those occupying a \emph{medial} position of $s$, i.e, one not corresponding to the first or last letter of $s$, when forming $\pi\in \mathcal{G}_n(\tau)$ from $\rho$.  The string in $\pi$ isomorphic to $\tau$ and containing the medial entries of $s$ then becomes the marked subword of $\pi$.

Thus, if $|\tau|=3$, with $\tau$ not containing a level, which is true of the four patterns under consideration here, there are $\binom{(n-j)+j-2}{j-2}=\binom{n-2}{j-2}$ possible $\pi$ that arise from each $\rho$.
Allowing $j$  to vary gives
\begin{equation}\label{Gntau}
|\mathcal{G}_n(\tau)|=\sum_{j=3}^n\binom{n-2}{j-2}|\mathcal{G}_j^*(\tau)|, \qquad n \geq 3,
\end{equation}
and thus we seek to enumerate the members of $\mathcal{G}_j^*(\tau)$ for $j \geq 3$ and the various $\tau$.

To aid in doing so, it is useful to differentiate between certain kinds of occurrences of a consecutive pattern $\tau$.  By a \emph{terminal} occurrence of $\tau$ within a word $w$, we mean one in which the final letter in the occurrence coincides with the final letter of $w$, with all other occurrences of $\tau$ being \emph{non-terminal}. Similarly, we will use the terms \emph{initial} and \emph{non-initial} when discerning occurrences of $\tau$ that start off the word $w$ from those that do not. \medskip

\noindent {\bf $\text{tot}_n(212)$:}  Let $J_n=\mathcal{G}_n^*(212)$ for $n\geq 3$.  Then $J_n$ is empty if $n=3$, so we seek to determine $|J_n|$ for $n\geq 4$.  To do so, first note that by the bijection used in the proof of  Lemma \ref{11avoidlem}, we have that members of $J_n$ are in one-to-one correspondence with ``marked'' sequences in $A_n$ wherein an occurrence of $010$ is marked. There are $(n-4)2^{n-5}$ such sequences $\pi=\pi_1\cdots\pi_{n-1}$ in $A_n$ for which the marked 010 is non-initial, as there are $n-4$ choices for the index $i\in [2,n-3]$ corresponding to the marked string $\pi_i\pi_{i+1}\pi_{i+2}$ of $\pi$ and $2^{n-5}$ choices concerning the remaining entries $\pi_j$, where $j \in [n-1]-\{1,i,i+1,i+2\}$. If the marked 010 is initial, then there are $2^{n-4}$ possibilities for $\pi$, as the entries in $\pi_4\cdots\pi_{n-1}$ may be chosen arbitrarily. Combining the two cases gives $(n-4)2^{n-5}+2^{n-4}=(n-2)2^{n-5}$ marked members of $A_n$, and hence the cardinality of $J_n$.  Thus, by
\eqref{Gntau}, we have
$$\text{tot}_n(212)=|\mathcal{G}_n(212)|=\sum_{j=4}^n(j-2)\binom{n-2}{j-2}2^{j-5}=\frac{n-2}{4}\sum_{j=4}^n\binom{n-3}{j-3}2^{j-3}=\frac{(n-2)(3^{n-3}-1)}{4},$$
as desired. \medskip

\noindent {\bf $\text{tot}_n(121)$:} First note that each $\alpha\in\mathcal{G}_n^*(212)$ gives rise to a $\beta \in \mathcal{G}_n^*(121)$ in which a non-terminal occurrence of $121$ is marked simply by shifting the marked string in $\alpha$ to the left by one position so as to obtain $\beta$.  This follows from the fact that the predecessor of the first $a$ in a string of the form $a,a-1,a$ within a member of $\mathcal{F}_n(11)$ must be $a-1$ since levels are not allowed and neither are consecutive descents.  Further, there are $2^{n-4}$ members of $\mathcal{G}_n^*(121)$ in which the marked 121 is terminal where $n \geq 4$, by Lemma \ref{11avoidlem}, upon taking an arbitrary $\lambda \in \mathcal{F}_{n-2}(11)$ and appending $b+1,b$ to $\lambda$, where $b$ denotes the final letter of $\lambda$. If $n=3$, then $\mathcal{G}_3^*(121)$ is the singleton consisting of 121 whose letters are marked.  Hence, we have
$$|\mathcal{G}_n^*(121)|-|\mathcal{G}_n^*(212)|=
    \begin{cases}
        2^{n-4}, & \text{if } n\geq 4;\\
        1, & \text{if } n=3.
    \end{cases}$$
By \eqref{Gntau}, this implies
\begin{align*}
\text{tot}_n(121)-\text{tot}_n(212)&=|\mathcal{G}_n(121)|-|\mathcal{G}_n(212)|=n-2+\sum_{j=4}^n\binom{n-2}{j-2}2^{j-4}\\
&=n-2+\frac{3^{n-2}-2n+3}{4}=\frac{3^{n-2}+2n-5}{4}.
\end{align*}
Adding this difference to the expression found above for $\text{tot}_n(212)$ yields the desired formula for $\text{tot}_n(121)$. \medskip

\noindent {\bf $\text{tot}_n(312)$:} Note first that members of $\mathcal{G}_n^*(312)$ are synonymous with ``marked'' sequences in $A_n$ wherein a 110 is marked.  Further, the cardinality of the latter set is seen to be $(n-4)2^{n-5}$ for $n\geq 5$, as four of the entries in $\pi=\pi_1\cdots \pi_{n-1} \in A_n$ are specified (namely, $\pi_1$, $\pi_i$, $\pi_{i+1}$, $\pi_{i+2}$ for some $i\in[2,n-3]$, where $\pi_i\pi_{i+1}\pi_{i+2}$ is the marked 110), with the others being chosen arbitrarily.  By \eqref{Gntau}, we then get
\begin{align*}
\text{tot}_n(312)&=\sum_{j=5}^n(j-4)\binom{n-2}{j-2}2^{j-5}=\text{tot}_n(212)-\sum_{j=4}^n\binom{n-2}{j-2}2^{j-4}\\
&=\frac{(n-2)(3^{n-3}-1)}{4}-\frac{3^{n-2}-2n+3}{4}=\frac{(n-5)3^{n-3}+n-1}{4},
\end{align*}
as desired. \medskip

\noindent{\bf $\text{tot}_n(123)$:} We first note that each maximal increasing run $0^a1^b$ in $\pi \in A_n$ where $a,b\geq 1$ is seen to give rise to $a+b-2$ occurrences of 123 in $\pi'\in\mathcal{F}_n(11)$, where $\pi \mapsto \pi'$ is the bijection from the proof of Lemma \ref{11avoidlem}.  Further, a terminal sequence $0^c$ yields $c-1$ additional occurrences of 123 in $\pi'$.  Thus, if $\pi$ is of the form
$$\pi=0^{a_1}1^{b_1}\cdots0^{a_r}1^{b_r}, \quad \text{or} \quad \pi=0^{a_1}1^{b_1}\cdots0^{a_r}1^{b_r}0^{a_{r+1}},$$
for some $r \geq 0$, where all exponents are positive, then $\sum_{i=1}^r(a_i+b_i)=n-1$ or $a_{r+1}+\sum_{i=1}^{r}(a_i+b_i)=n-1$ implies there are
$$\sum_{i=1}^r(a_i+b_i-2)=n-1-2r, \quad \text{or} \quad a_{r+1}-1+\sum_{i=1}^r(a_i+b_i-2)=n-1-(2r+1)$$
occurrences of $123$ in all in $\pi'$.  Therefore, upon combining cases, we have $\#123(\pi')=n-1-\text{run}(\pi)$ for all $\pi \in A_n$, where $\text{run}(\pi)$ denotes the number of runs of letters in $\pi$.

Note that there are $\binom{n-2}{\ell-1}$ members $\pi \in A_n$ with $\text{run}(\pi)=\ell$, and hence the same number of $\rho \in \mathcal{F}_n(11)$ with $\#123(\rho)=n-\ell-1$ for each $\ell \in [n-1]$. This implies
$$|\mathcal{G}_n^*(123)|=\sum_{\ell=1}^{n-2}\binom{n-2}{\ell-1}(n-\ell-1)=(n-2)\sum_{\ell=1}^{n-2}\binom{n-3}{\ell-1}=(n-2)2^{n-3}, \quad n \geq 3.$$
By \eqref{Gntau}, we then have
$$\text{tot}_n(123)=\sum_{j=3}^n(j-2)\binom{n-2}{j-2}2^{j-3}=(n-2)\sum_{j=3}^n\binom{n-3}{j-3}2^{j-3}=(n-2)3^{n-3},$$
which completes the proof. \hfill \qed \medskip

\noindent\emph{Proofs of formulas for $f_n(\tau)$ for $\tau\in\{121,123,212,312\}$:} \medskip

\noindent{\bf $f_n(121)$:} Given $j \in [n]$, $0 \leq p \leq \lfloor(j-1)/2\rfloor$ and $0 \leq r \leq p$, let $A_j^{(p,r)}$ denote the subset of $A_j$ whose members contain $p$ runs of 1 with $r$ of these runs of length one, where $A_j$ is as in the proof of Lemma \ref{11avoidlem}.  That is, $\alpha \in A_j^{(p,r)}$ implies $\alpha=0^{a_1}1^{b_1}\cdots0^{a_p}1^{b_p}0^{a_{p+1}}$ wherein $a_{p+1}+\sum_{i=1}^{p}(a_i+b_i)=j-1$, each exponent $a_i,b_i\geq1$ for $i \in [p]$, with $a_{p+1}\geq0$, and $b_i=1$ for exactly $r$ indices $i \in [p]$.
Thus, members of $A_j^{(p,r)}$ may be identified as integral sequences $(x_1,x_2,\ldots,x_{2p+1})$ summing to $j-1$ such that $x_i \geq 1$ for all $i\in[2p]$ with $x_{2p+1}\geq0$ and exactly $r$ of the components $x_{2j}$ for $1 \leq j \leq p$ equal to one.  We first select these $r$ components $x_{2j}$ in one of $\binom{p}{r}$ ways.  Then there are $\binom{j-3p+r-1+(2p-r)}{2p-r}=\binom{j-p-1}{2p-r}$ possibilities for the remaining $2p-r+1$ components, as they are seen to be synonymous with non-negative integral solutions to
$$y_1+y_2+\cdots+y_{2p-r+1}=j-1-2p-(p-r)=j-3p+r-1,$$ where the extra subtracted term $p-r$ accounts for the $p-r$ unselected components $x_{2j}$ for $j \in [p]$ each of which must be at least two.

We thus have $|A_j^{(p,r)}|=\binom{p}{r}\binom{j-p-1}{2p-r}$, with members of $A_j^{(p,r)}$ corresponding under the mapping $\pi \mapsto \pi'$ to the subset $S_j^{(p,r)}$ of $\mathcal{F}_j(11)$ whose members contain $p$ descents and $r$ occurrences of 121.  In forming $\pi \in \mathcal{F}_n(121)$ from a skeleton $\rho \in \mathcal{F}_j(11)$ by adding $n-j$ letters as described previously, one must add at least one duplicate to follow any letter playing the role of the `2' within an occurrence of 121 in $\rho$.  Note that if $\pi$ were to contain an occurrence of 121, then so must its skeleton, and hence destroying each occurrence of 121 of $\rho$ by duplicating the middle letter ensures $\pi$ avoids 121.  Let $\text{skel}(\pi)$ denote the skeleton of $\pi \in \mathcal{F}_n$.  For each $\rho \in S_j^{(p,r)}$, there are thus $\binom{(n-j-r)+j-1}{j-1}=\binom{n-r-1}{j-1}$ possible $\pi \in \mathcal{F}_n(121)$ for which $\text{skel}(\pi)=\rho$, since $r$ of the $n-j$ letters to be added have locations which are specified.  Considering all possible $j$, $p$ and $r$ then implies the formula for $f_n(121)$ stated in Corollary \ref{Cxythc2}. \medskip

\noindent{\bf $f_n(123)$:} In the proof given above for $\text{tot}_n(123)$, we saw that there were $\binom{j-2}{\ell-1}$ members $\rho \in \mathcal{F}_j(11)$ with $\#123(\rho)=j-\ell-1$ for each $\ell \in [j-1]$, where $j \geq 2$. Note that in forming $\pi \in \mathcal{F}_n(123)$ from such a skeleton $\rho$,  each of the $j-\ell-1$ letters that play the role of `2' in an occurrence of 123 within $\rho$ must have at least one duplicate letter added so as to directly follow it in order to ensure that $\pi$ avoids 123.  Thus, for each such $\rho$, there are $\binom{n-j-(j-\ell-1)+j-1}{j-1}=\binom{n-j+\ell}{j-1}$ possible $\pi$ for which $\text{skel}(\pi)=\rho$.  Considering all $j$ and $\ell$ gives
$$f_n(123)=1+\sum_{j=2}^n\sum_{\ell=1}^{j-1}\binom{j-2}{\ell-1}\binom{n-j+\ell}{j-1}, \qquad n \geq 1,$$
as desired, where the initial $+1$ accounts for the member of $\mathcal{F}_n(123)$ consisting of all $1$'s. \medskip

\noindent{\bf $f_n(212)$ and $f_n(312)$:} We first determine $f_n(212)$.  Given $j \geq 2$, $1 \leq p \leq \lfloor j/2 \rfloor$ and $0 \leq r \leq p-1$, let $A_j(p,r)$ denote the subset of $A_j$ whose members contain $p$ runs of $0$ and $p-r-1$ non-terminal runs of 1 of length one.  That is, $\pi \in A_j(p,r)$ implies it can be expressed as $\pi=0^{a_1}1^{b_1}\cdots0^{a_p}1^{b_p}$ such that $\sum_{i=1}^p(a_i+b_i)=j-1$ where each $a_i,b_i\geq 1$, with the possible exception of $b_p$, which could be zero as well, and $b_i=1$ for exactly $p-r-1$ indices $i \in [p-1]$.  Once these $p-r-1$ indices have been specified (in one of $\binom{p-1}{r}$ ways), then the sequence of remaining exponents $a_i,b_i$ may be viewed as a non-negative integral solution to
$$y_1+y_2+\cdots+y_{p+r+1}=j-1-(p-r-1)-2r-p=j-2p-r,$$
where the subtracted $p-r-1$ accounts for the indices $i\in[p-1]$ such that $b_i=1$, the subtracted $2r$ for the $r$ indices $i<p$ such that $b_i\geq 2$ and the subtracted $p$ for the $a_i\geq 1$ where $i \in [p]$.  Thus, there are $\binom{(j-2p-r)+p+r}{p+r}=\binom{j-p}{p+r}$ possibilities concerning the remaining $p+r+1$ exponents once the $p-r-1$ indices $i<p$ for which $b_i=1$ have been selected, and hence
$|A_j(p,r)|=\binom{p-1}{r}\binom{j-p}{p+r}$.

Note that members of $A_j(p,r)$ correspond to the subset $T_j^{(p,r)}$ of $\mathcal{F}_j(11)$ whose members contain $p-r-1$ occurrences of 212 and $r$ occurrences of 312.  Let $\rho \in T_j^{(p,r)}$ and we wish enumerate $\pi \in \mathcal{F}_n(212)$ such that $\text{skel}(\pi)=\rho$.  In order to ensure that $\pi$ avoids 212 when forming $\pi$ from $\rho$, we must insert at least one copy to directly follow each of the $p-r-1$ letters playing the role of `1' in an occurrence of 212 within $\rho$.  Thus, there are $\binom{n-j-(p-r-1)+j-1}{j-1}=\binom{n-p+r}{j-1}$ members $\pi \in \mathcal{F}_n(212)$ such that $\text{skel}(\pi)=\rho$ for each $\rho \in T_j^{(p,r)}$.  Considering all possible $j$, $p$ and $r$ yields the desired formula
$$f_n(212)=1+\sum_{j=2}^n\sum_{p=1}^{\lfloor j/2\rfloor}\sum_{r=0}^{p-1}\binom{p-1}{r}\binom{j-p}{p+r}\binom{n-p+r}{j-1}, \qquad n \geq 1,$$
where the initial $+1$ accounts for the all $1$'s member of $\mathcal{F}_n(212)$.  The same basic proof applies to finding $f_n(312)$ except now there are $\binom{n-r-1}{j-1}$ members $\pi \in \mathcal{F}_n(312)$ such that $\text{skel}(\pi)=\rho$ for each $\rho \in T_j^{(p,r)}$.  \hfill \qed \medskip

\noindent\emph{Proofs of the equivalences $112\approx122$ and $211\approx221$:} \medskip

We decompose $\pi \in \mathcal{F}_n$ into its increasing blocks as $\pi=I_1\cdots I_\ell$ for some $\ell \geq 1$.  Suppose $I_j=a_1^{r_1}a_2^{r_2}\cdots a_k^{r_k}$ where $j \in [\ell]$, with $a_{i+1}=a_i+1$ for each $1\leq i \leq k-1$ and $r_i\geq 1$ for all $i$. Define $\widetilde{I}_j=a_1^{r_k}a_2^{r_{k-1}}\cdots a_k^{r_1}$, i.e., we reverse the sequence of multiplicities of the letters  in $I_j$ to obtain $\widetilde{I}_j$.  One may verify that an occurrence of 112 in $I_j$ between $a_i$ and $a_{i+1}$ for some $i \in [k-1]$ corresponds to an occurrence of 122 between $a_{k-i}$ and $a_{k-i+1}$ in $\widetilde{I}_j$, and vice versa. Let $\widetilde{\pi}=\widetilde{I}_1\cdots\widetilde{I}_\ell$.  As the number of $112$'s in $I_j$ equals the number of $122$'s in $\widetilde{I}_j$ for each $j \in [\ell]$, we have $\#112(\pi)=\#122(\widetilde{\pi})$ for all $\pi \in \mathcal{F}_n$.  Since $\pi \mapsto \widetilde{\pi}$ is a bijection on $\mathcal{F}_n$, the first equivalence is established.

For the second, consider a (maximal) sequence of the form $a^xb^y$ within $\rho \in \mathcal{F}_n$, where $a>b$ and $x,y \geq 1$.  Then replace with $a^yb^x$ and repeat for each such sequence to obtain $\widehat{\rho}$.  Since $\rho$ being a member of $\mathcal{F}_n$ cannot contain two consecutive descents, such sequences $a^xb^y$ within $\rho$ are mutually disjoint from one another. Hence, the mapping $\rho \mapsto \widehat{\rho}$ is a bijection.  Moreover, each occurrence of 211 in $\rho$ is seen to give rise to a distinct occurrence of 221 in $\widehat{\rho}$, and conversely.  This implies $\#211(\rho)=\#221(\widehat{\rho})$ for all $\rho$, which establishes the second equivalence and completes the proof.  Furthermore, the bijections $\pi \mapsto \widetilde{\pi}$ and $\rho \mapsto \widehat{\rho}$ demonstrate that the joint distributions of 112/122 and of 211/221 on $\mathcal{F}_n$ are both symmetric. \hfill \qed

\end{document}